\documentclass[a4paper,11pt]{article}

\usepackage{geometry}
\usepackage{amsmath}
\usepackage[parfill]{parskip}    % Activate to begin paragraphs with an empty line rather than an indent
\usepackage{graphicx}
\usepackage{amssymb}
\usepackage{float}
\usepackage{caption}
\usepackage{hyperref}
\usepackage{mathtools}
\usepackage{color}
\usepackage{listings}

\usepackage[normalem]{ulem}
\usepackage{ mathrsfs }
\usepackage{bbm}
\usepackage[dvipsnames]{xcolor}
\usepackage{todonotes}
\usepackage{enumitem}
\usepackage{units}

\usepackage{amsthm}
\usepackage{thmtools}
\usepackage{thm-restate}

\usepackage[nottoc,notlof]{tocbibind}

\newtheorem{theorem}{Theorem}[section]
\newtheorem{corollary}[theorem]{Corollary}
\newtheorem{proposition}[theorem]{Proposition}
\newtheorem{claim}{Claim}
\newtheorem{lemma}[theorem]{Lemma}

\newtheorem*{theorem*}{Theorem}
\newtheorem{assumption}[theorem]{Assumptions}

\theoremstyle{remark}
\newtheorem{remark}[theorem]{Remark}
\newtheorem{example}[theorem]{Example}

\theoremstyle{definition}
\newtheorem{definition}[theorem]{Definition}

\usepackage{ wasysym }
\usepackage{ textcomp }

\DeclareMathOperator*{\argmax}{arg\,max}

\newcommand{\cE}{\mathcal{E}}
\newcommand{\cF}{\mathcal{F}}

\newcommand{\de}{\operatorname{d}}

\newcommand{\sP}{\mathscr{P}}

\DeclareMathOperator{\p}{\mathbb{P}}

\usepackage[affil-sl]{authblk}
\title{Simultaneous Cutoff on the Multitype Configuration Model}
\author[1]{John Fernley}
\author[1,2]{Bal\'azs Gerencs\'er}
\affil[1]{HUN-REN Alfr\'ed R\'enyi Institute of Mathematics,
            Budapest, % 1053,
            Hungary}
\affil[2]{ELTE E\"otv\"os Lor\'and University,
            Budapest, % 1117,
            Hungary}
          \affil[ ]{\small\{fernley.john, gerencser.balazs\}@\hspace{0.1em}renyi.hu}
\date{\today}

\usepackage{amssymb}

\usepackage{graphicx}

\usepackage{appendix}
\usepackage{placeins}

\newcommand{\D}{\mathbb{D}}
\newcommand{\al}{\alpha}
\newcommand{\Pab}{\mathscr{P}_A}

\usepackage[style=alphabetic,isbn=false,doi=false,maxnames=10,backend=biber]{biblatex}
\renewbibmacro{in:}{}
\DeclareFieldFormat
  [article,inbook,incollection,inproceedings,patent,thesis,unpublished]
  {title}{#1\isdot}
\AtEveryBibitem{%
  \clearlist{language}%
}
\usepackage[bblfile=multiplex]{biblatex-readbbl}

\begin{document}

\maketitle
%\begin{center}
%{\LARGE  Optimal Mixing in Multiplex Networks}\\
%\vspace{1.5em}
%{John Fernley\footnote{\label{address}Alfr\'ed R\'enyi Institute of Mathematics,
%			Re\'altanoda utca 13-15,
%            Budapest,
%            1053, 
%            Hungary, {\tt fernley@renyi.hu}.}}
%            \qquad
%            Bal\'azs Gerencs\'er\cref{address}
%\end{center}

\begin{abstract}
We find Gaussian cutoff profiles for the total variation distance to stationarity of a random walk on a multiplex network: a finite number of directed configuration models sharing a vertex set, each with its own bounded degree distribution and edge probability. %By applying a Markov chain version of the Berry-Esseen theorem from the literature. 
Further we consider the minimal total variation distance over this space of possible doubly stochastic edge probabilities at each point in time. 
Looking at all possible dynamics simultaneously on one realisation of the random graph, we find that this sequence of minimal distances converges in probability to the same cutoff profile as the chain with entropy maximising transition probabilities.

%In part this is a generalisation of \cite{bordenave18} with estimates made uniform, and where the Gaussian shape instead comes from a Markov chain version of the Berry-Esseen theorem. 

\vspace{1.5em}
\center{\tiny%\centering
%\noindent

%\noindent
%\centering
}

\end{abstract}

%\doublespacing
%\tableofcontents   
%\onehalfspacing

%refer to description items
\makeatletter
\let\orgdescriptionlabel\descriptionlabel
\renewcommand*{\descriptionlabel}[1]{%
  \let\orglabel\label
  \let\label\@gobble
  \phantomsection
  \edef\@currentlabel{#1\unskip}%
  \let\label\orglabel
  \orgdescriptionlabel{#1}%
}
\makeatother

\section{Introduction}

A multiplex network on vertex set $[N]=\{1,\dots,N\}$ has $I\geq 2$ \emph{layers} containing different types of connectivity information \cite{bianconi2013statistical}. 
We construct such a network through a directed configuration model defined by a set $\D$ of $I$-layer directed vertex degrees $\al \in [\Delta]^{2I}$ (so crucially with $\min\al\geq 1$):
\[
\alpha=(\al_1^+,\al_1^-,\dots,\al_I^+,\al_I^-) ,
\]
each vertex $v\in [N]$ having a type $\tau(v) \in \D$, leading to a frequency in the network $k(\al)=\sum_{v =1}^N \mathbbm{1}_{\tau(v)=\alpha}$. To construct a multiplex configuration model, we require
\[
\sum_{\al \in \D} k(\al)\al_i^+
=\sum_{\al \in \D} k(\al)\al_i^-
\]
so that by sampling an independent uniform bijection from heads to tails for each layer $i \in [I]$ we produce the directed edges of each graph $(A^{(i)}_{vw})_{v,w \in [N]}$. Essentially, what we have constructed is $I$ independent directed configuration models (each defining a layer) with a prescribed identification of the vertices in each.

Each layer of the multiplex network represents a different type of connection and so to a directed edge in layer $i \in [I]$ we attach probability $p_i$ for the walker. To obtain a Markov process, this requires
\[
\forall \al \in \D, \quad
\sum_{i \in [I]} p_i \al_i^+=1
\] 
and we will further restrict our attention to the family of doubly stochastic walks
\[
\forall \al \in \D, \quad
\sum_{i \in [I]} p_i \al_i^-=1
\]
so that we can compare mixing speeds to the same stationary distribution which is, on the high probability event of a strongly connected graph, the uniform $\nicefrac{1}{N}$. This suggests the stationary \emph{layer} of a walk on the multiplex network has distribution
\[
\pi(i)=  \sum_{\al \in \D} \frac{k(\al)}{N} \al_i^+ p_i
\]
and this is also the stationary distribution of the layer Markov chain $(X_k)_k$ which is an approximation to the true %ignoring coupling details, the ``annealed'' version of the
layer walk, with transition matrix
\begin{equation}\label{eq_layer_chain}
L_{ij}=
\frac{\sum_{\al \in \D} k(\al) \al_i^- \al_j^+ p_j}{\sum_{\al \in \D} k(\al) \al_i^-}.
\end{equation}

Mixing of the walker on a heuristic level requires that the walker takes a path which had probability approximately $\nicefrac{1}{N}$. 
If this Markov chain $L$ is ergodic, the path probability at time $t$ should decay like
\[
e^{t \,\mathbb{E}_{\pi}(\log p_i)}
\]
which is why we define
\begin{equation}\label{eq_mu_def}
\mu_p=\mathbb{E}_{\pi}\left(\log \frac{1}{p_i} \right),
\qquad
t_p=\frac{\log N}{\mu_p}
\end{equation}
as the entropy rate and cutoff time. 
Zooming in to see the correct cutoff timescale will require a central limit approximation, for which we need to understand the dynamical variance $\sigma^2_p$ and the associated window scale
\begin{equation}\label{eq_sigma_def}
\frac{1}{n}\mathbb{V}{\rm ar}\left( \sum_{k=1}^n \log p_{X_k} \right) 
\stackrel{\mathbb{P}}{\rightarrow}
\sigma^2_p,
\qquad
w_p=\sigma_p \sqrt{\frac{\log N}{\mu_p^3}}.
\end{equation}

This dynamical variance is necessarily the least explicit parameter, but we know the limit exists from the alternative definition of \cite{kloeckner2019} and will see in Lemma  \ref{lemma_dynamical_variance_bounds} that it can be loosely approximated by the stationary variance $\mathbb{V}{\rm ar}_\pi \log p_i$.
%\[
%\exists \al, \beta \in \D : \, \de^+(\al)\neq\de^+(\beta) \text{ or } \de^-(\al)\neq\de^-(\beta).
%\]

%Finally, we further restrict that
%\[
%\forall i \in [I], \, \exists \al \in \D : \quad \alpha_i^+ \neq 1 \text{ or }  \alpha_i^- \neq 1
%\]
%to exclude any degenerate layers in the configuration model consisting only of cycles.

\subsubsection*{Notation}

We use the standard Landau notation e.g. $f=O(g)$, with subscript $f=O_{\p}(g)$ if the inequality defining the order holds with high probability as $N\rightarrow\infty$.
%and superscript $f=O^{\log N}_{\p}(g)$ if there is additionally a polylogarithmic factor correction. That is,
%\[
%f(N)=O^{\log N}_{\p}(g(N))
%\iff
%\exists C>0:
%\p\left(
%f(N)\leq g(N) \log^C N
%\right)\rightarrow 1.
%\]
We further use the following notation to neglect smaller orders
\[
f(N) \lesssim g(N)
\iff
f(N)
\leq
g(N)
(1+o(1)).
\]

%\clearpage
\section{Main Results}

%\JFtodo{mention several places that everything is for large $N$}

%\begin{definition}
%\[
%\left[P^t\right]_{v\rightarrow w}
%=
%\sum_{i=1}^I
%\sum_{j=1}^I
%\sum_{\mathfrak{p}\in \mathfrak{P}^t_{(v,i)(w,j)}}
%\mathbf{w}(\mathfrak{p})
%\]
%\end{definition}

N is taken sufficiently large in all results. The walk on $[N]$ which picks each out-edge of layer $i$ with probability $p_i$ has a transition matrix which we write as $P$, and then our interest is in the total variation convergence of this walk.

\begin{definition}
For the distance and vector
\[
\| v \|_{\rm TV}=\frac{1}{2}\sum_{i \in [N]}|v_i|,
\qquad
\pi=\frac{1}{N}\mathbbm{1_{[N]}},
\]
we measure the distance from stationarity with
\[
\mathcal{D}_{v}(t)=
\left\|
P^t(v,\cdot)-\pi
\right\|_{\text{TV}},
\qquad
\mathcal{D}(t)=\max_v \mathcal{D}_{v}(t \vee 0).
\]
\end{definition}

To avoid the case of constant edge probability, where the layer structure is trivialised, we will impose that the configuration model in question doesn't have all outdegrees and indegrees equal to the same constant.

\begin{assumption}\label{assumptions}
\begin{enumerate}[label={(\alph*)},ref={\theproposition~(\alph*)}]
\item There exists $p$ with $0< p \in \mathscr{P}$ (a positive solution to \eqref{eq_p_def}).
\item \label{ass_different_outdegree}
At least one pair of types $\al, \beta \in \D$ has either differing total outdegree $\sum \al^+ \neq \sum\beta^+$ or indegree $\sum \al^- \neq \sum\beta^-$.
\end{enumerate}
\end{assumption}

\begin{remark}
In the absence of Assumption \ref{ass_different_outdegree}, the constant vector is always a solution for $p$ and by differentiating $\mu$ we can verify that it's a stationary point in the directions permitted in $\mathscr{P}$.
\end{remark}

\begin{restatable}{theorem}{optimisercutoff}
\label{thm_optimiser_cutoff}
Write $w_\bigstar$ and $t_\bigstar$ for the shape parameters of \eqref{eq_mu_def} and \eqref{eq_sigma_def} with probabilities
\[
p_\bigstar:=\argmax_{p \in \sP} \mu_p,
\]
then with Assumptions \ref{assumptions}
\[
%\lim_{N \rightarrow \infty}
\sup_{t >0}
\left|
\inf_{p \in \sP}
\mathcal{D}(t)-1+\Phi\left(
\frac{t-t_\bigstar}{w_\bigstar}
\right)
\right|
\stackrel{\p}{\rightarrow}
 0.
%\lesssim
%\frac{2 \left( \log \log N \right)^2 \sqrt{\log \Delta I}}{ \log^{\nicefrac{1}{3}} N}
\]
\end{restatable}

To prove the above result we follow an approach similar to \cite{bordenave18}, but where we write
\[
p_*=\min_{i \in [I]}p_i, \qquad
p^*=\max_{i \in [I]}p_i, 
\]
and look for some $A\geq 1$ at the space
\[
\Pab:=\left\{
p \in \mathscr{P}: \,
p_* \geq A \log^{-1/3} N
\right\}
\]
to show the following uniform (over $\Pab$) version of their cutoff result for the simple random walk. Note also that the error probability has no graph parameters so this result is also uniform over all valid configuration models that we could have chosen.

%\begin{theorem}\label{thm_cleaner_simultaneous_cutoff}
%Assume Assumptions \ref{assumptions} and take some constant $A \geq 1$.
%Then for the shape parameters of \eqref{eq_mu_def} and \eqref{eq_sigma_def} we find a constant $C=C(\Delta,I)$ (depending on the maximal degree and the number of layers) such that
%\[
%%\lim_{N \rightarrow \infty}
%\sup_{p \in \Pab}
%\sup_{\lambda \in \mathbb{R}}
%\left|\mathcal{D}(t_p+\lambda w_p)-1+\Phi(\lambda)\right|
%\leq\frac{C}{A^{\nicefrac{3}{2}}}
%\]
%with probability $1-7e^{-\sqrt[3]{\log N}}$.
%\end{theorem}
\begin{theorem}\label{thm_cleaner_simultaneous_cutoff}
Under Assumptions \ref{assumptions}, let $A=A(N)\rightarrow \infty$ sufficiently slowly. 
Then for the shape parameters of \eqref{eq_mu_def} and \eqref{eq_sigma_def} we find
\[
\sup_{p \in \Pab}
\sup_{\lambda \in \mathbb{R}}
\left|\mathcal{D}(t_p+\lambda w_p)-1+\Phi(\lambda)\right|
=O
\left(
\frac{1}{A^{\nicefrac{3}{2}}}
\right)
\]
on an event of probability $1-7e^{-\sqrt[3]{\log N}}$.
\end{theorem}
This is proved as Theorem \ref{thm_simultaneous_cutoff} which is a stronger variant but harder to parse; Theorem \ref{thm_simultaneous_cutoff} also demonstrates that any $A=O(\log^{\nicefrac{1}{5}} N)$ is sufficiently slow for the given order. Outside of $\Pab$, we have the following result (proved as Claims 3 and 4 of Theorem \ref{thm_optimiser_cutoff}) which covers the rest of the space of possible walks, again evidently with uniform control over the space of configuration models.

\begin{proposition}\label{prop_bad_set_control}
For any sequences $t \sim t_\bigstar$ and $A\rightarrow \infty$ sufficiently slowly, and constant $\delta>0$,
\[
\sup_{p \notin \Pab}
\mathcal{D}(t)
\geq
1-\delta
\]
for $N$ large enough, with probability $1-7e^{-\sqrt[3]{\log N}}$.
\end{proposition}

From this, together with Theorem \ref{thm_cleaner_simultaneous_cutoff}, we could write a version of Theorem \ref{thm_optimiser_cutoff} which is also uniform in the configuration model parameters.

\begin{example}\label{example}

In the two-layer configuration model defined by
\[
\alpha^+=\alpha^-=(2,3,1)
\]
\[
\beta^+=\beta^-=(2,1,2)
\]
\[
k(\alpha)=k(\beta)=\frac{N}{2}
\]
we have a $1$-dimensional space of doubly stochastic dynamics
\[
\sP=\left\{
\lambda
\left(
\frac{1}{2} ,
0,
0
\right)
+(1- \lambda)
\left(
0,
\frac{1}{5},
\frac{2}{5}
\right)
: \lambda \in [0,1]
\right\}
\]
and the optimiser $p_\bigstar$ in this space is at the parameter 
$
\lambda_\bigstar=\frac{1}{1+\nicefrac{5}{2^{\nicefrac{8}{5}}}}
$
 which gives:
%\[
%\sigma^2_\bigstar
%=
%3 \log ^2\left(1+\frac{5}{\sqrt[5]{8}}\right)
%-\frac{
%\left(4+\frac{19}{5} \sqrt[5]{4}\right) \log ^2(2)}{4+5 \sqrt[5]{4}}
% -2 \log (2) \log \left(1+\frac{5}{\sqrt[5]{8}}\right)\approx 3.94983.
%\]
\[
\mu_\bigstar=
\log \left(2+\frac{5}{2^{\nicefrac{3}{5}}}\right)\approx 1.66747;
\qquad
\sigma^2_\bigstar
=
\frac{6}{5}2^{\nicefrac{2}{5}}\log^2 2
 \approx 0.760754.
\]

\begin{figure}[h]
\centering
\includegraphics[scale=0.8]{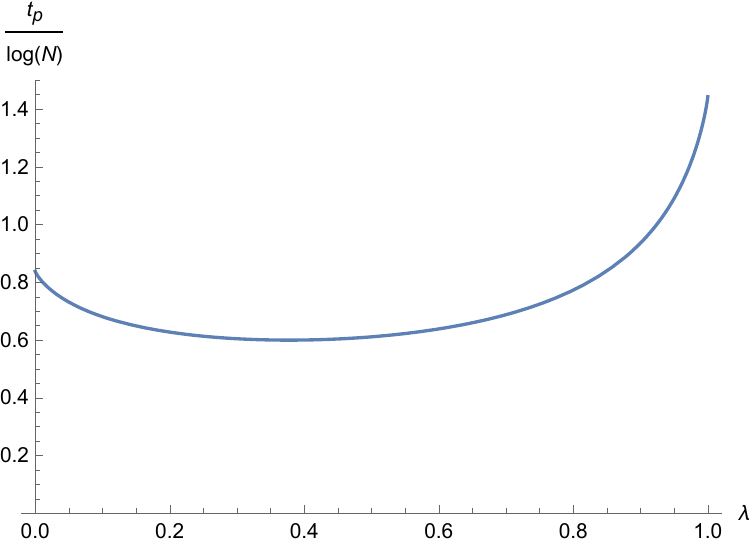}
\caption{
The asymptotic constant multiple of $\log N$ of the cutoff times of chains in Example \ref{example}.
}
\end{figure}

Then by Theorem \ref{thm_optimiser_cutoff}
%\[
%\sup_{t >0}
%\left|
%\inf_{p \in \mathcal{P}}
%\mathcal{D}(t)-1+\Phi\left(
%\frac{\left(2t \log \left(2+\frac{5}{\sqrt[5]{8}}\right)-\log N\right)\sqrt{2 \log \left(2+\frac{5}{\sqrt[5]{8}}\right)}}{\sqrt{\left( 3 \log ^2\left(1+\frac{5}{\sqrt[5]{8}}\right)
%-\frac{
%\left(4+\frac{19}{5} \sqrt[5]{4}\right) \log ^2(2)}{4+5 \sqrt[5]{4}}
% -2 \log (2) \log \left(1+\frac{5}{\sqrt[5]{8}}\right) \right)\log N}}
%\right)
%\right|\rightarrow 0
%\]
\[
\sup_{t >0}
\left|
\inf_{p \in \sP}
\mathcal{D}(t)-1+\Phi\left(
\frac{t \mu_\bigstar -\log N}{\sigma_\bigstar}
\sqrt{\frac{\mu_\bigstar}{\log N}}
\right)
\right|\rightarrow 0
\]
or with approximate constants
\[
\inf_{p \in \sP}
\mathcal{D}(t)
\approx
1-\Phi\left(
2.4687 \frac{t}{\sqrt{\log N}} - 1.4805 \sqrt{\log N}
\right).
\]
\end{example}

\FloatBarrier
%\afterpage{\FloatBarrier}
\subsubsection*{Organisation of Paper}

First in Section \ref{sec_approximation} we collect preliminary results connecting the walker on the network to the approximating layer chain on the type space $[I]$, as well as for approximating one layer chain by another with similar layer probabilities. In Section \ref{sec_cutoff} we put these results together to get a Gaussian cutoff profile for a range of layer chains defined by a vector $p$ of edge probabilities, concluding with the main result Theorem \ref{thm_cleaner_simultaneous_cutoff} which gives this cutoff shape simultaneously for all these $p$, with high probability on the same random graph.

Finally, in Section \ref{sec_optimisation} we let $p=p_\bigstar$ be the optimiser and compare its shape to every other $q$ in the space, both the non-degenerate $q$ of Theorem \ref{thm_cleaner_simultaneous_cutoff} and the rest, for our main result Theorem \ref{thm_optimiser_cutoff} on the cutoff shape of the minimal distance.

%\JFtodo{conclusion?}

%\clearpage
\section{Approximation}\label{sec_approximation}

For the set of permissible $p$ we write
\begin{equation}\label{eq_p_def}
\mathscr{P}:=\left\{
p \in [0,1]^I: \,
\forall \al \in \D,  \,
\sum_{i \in [I]} p_i \al_i^+=\sum_{i \in [I]} p_i \al_i^-=1
\right\}
\end{equation}
and for technical reasons we will often further restrict to
\[
\Pab:=\left\{
p \in \mathscr{P}: \,
p_* \geq A \log^{-1/3} N
%\,
%\operatorname{Var}_\pi( \log p_i)\geq
%B
%\frac{ \left( \log \log N \right)^4 }{\log N}
\right\}
\]
for $A$ sufficiently large. With this notation, we will occasionally write stronger restrictions in the range $1\leq A\leq \theta \log^{1/3}N$ for some small constant $\theta$.

Divide the hypercube $(0,1]^I$ into a partition $G$ of $\log^{3I} N$ smaller hypercubes, each of edge length $\log^{-3} N$. Then give each of these ``grid elements'' $g \in G$ an arbitrary representative vector $p_g \in g \cap \Pab$ if that intersection is nonempty, to define a kind of quantising function ${ r} : \Pab \rightarrow \Pab$ with
\[
p \in g \implies
{ r}(p)=p_g.
\]

This allows the uncountable set of parameters to be controlled through a finite set of representatives, which we will see well approximate all relevant quantities of the walk of interest.

\begin{remark}\label{remark_rel_error}
For $p \in \Pab$ and $r$ defined above, we find relative error $\| 1-{ r}(p)/p  \|_\infty \leq \tfrac{1}{A}\log^{-\nicefrac{8}{3}}N$.
\end{remark}

The Dirichlet form of a chain $Q$ on $[n]$ with stationary measure $\pi$ is given by
\[
\mathcal{E}_Q(f,g):=\frac{1}{2}\sum_{i=1}^n\sum_{i=1}^n \pi(i)Q_{ij} (f(i)-f(j))(g(i)-g(j))
\]
which can, by \cite[Theorem 3.25]{aldous-fill-2014}, be used to give the following definition of the relaxation time.

\begin{definition}[Relaxation Time]\label{def_relaxation}
\[
t_{\rm rel}:=\sup_{x \in \mathbb{R}^n} \frac{\operatorname{Var}_\pi(f)}{\mathcal{E}_P(f,f)}.
\]
\end{definition}

This agrees with the usual definition (the inverse of the spectral gap of $P$) if $P$ is reversible, and when $P$ is not reversible it is instead the inverse of the spectral gap of the additive reversibilisation $\frac{P+P^*}{2}$.

We need the following contraction property to apply the results of \cite{kloeckner2019}.

\begin{corollary}\label{cor_contraction}
The dynamic \eqref{eq_layer_chain} with stationary layer distribution $\pi$ is a contraction in the norm
\begin{equation}\label{eq_norm_def}
\|f\|:= \|f\|_\infty + c \sqrt{\operatorname{Var}_\pi f}
\end{equation}
with  $c  =\nicefrac{1}{\sqrt{p_*}}$, 
in the sense that every $f \in \mathbb{R}^I$ with $\pi(f)=0$ has
\[
\frac{\|Lf\|}{\|f\|} \leq 1-\frac{1}{4\Delta}.
\]
\end{corollary}

\begin{proof}
We refer to Mihail's identity \cite[Lemma 1.13]{montenegro2005} and the expression of Definition \ref{def_relaxation} to show
\[
\operatorname{Var}_\pi(Lf)=\operatorname{Var}_\pi(f)-\mathcal{E}_{L^*L}(f,f)
\leq\left(
1-\frac{1}{t_{\rm rel}}
\right)\operatorname{Var}_\pi(f)
\]

We then bound the relaxation time by separation
\[
t_{\rm rel}=\max_{f \in \mathbb{R}^I}
\frac{\operatorname{Var}_\pi(f)}{\mathcal{E}_{L^*L}(f,f)}
=\max_{f \in \mathbb{R}^I}
\frac{\sum_i \sum_j \pi(i) \pi(j) (f_i-f_j)^2}{\sum_i \sum_j \pi(i) [L^*L]_{ij} (f_i-f_j)^2}
\leq
\max_{i,j \in [I]}
\frac{\pi(j)}{[L^*L]_{ij}}
\leq
\max_{i,j \in [I]}
\frac{\pi(j)}{L_{ij}}
\]
and the separation for $L$ is controlled, because $\alpha_i^- \geq 1$ for every $\al \in \D$ and $i \in [I]$, by
\begin{equation}\label{eq_separation}
L_{ij}=
\frac{\sum_{\al \in \D} k(\al) \al_i^- \al_j^+ p_j}{\sum_{\al \in \D} k(\al) \al_i^-}
\geq
\frac{\sum_{\al \in \D} k(\al) \al_j^+ p_j}{\sum_{\al \in \D} k(\al) \al_i^-}
=
\frac{\pi(j)}{\sum_{\al \in \D} \frac{k(\al)}{N} \al_i^-}
\geq
\frac{\pi(j)}{\Delta}.
\end{equation}

%\[
%\sqrt{\frac{\operatorname{Var}_\pi(Px)}{\operatorname{Var}_\pi(x)}} \leq
%\sqrt{1-\frac{2\alpha}{t_{\rm rel}}}
%\leq
%1-\frac{\alpha}{t_{\rm rel}}
%\implies
%\frac{\|Px\|}{\|x\|} \leq \frac{\|Px\|_\infty + (1-\nicefrac{\alpha}{t_{\rm rel}})\sqrt{\operatorname{Var}_\pi(x)}}{\|x\|_\infty + \sqrt{\operatorname{Var}_\pi(x)}}.
%\]

If the mean $\pi(f)=0$, the variance is just the weighted 2-norm
\[
\operatorname{Var}_\pi(f)=\|f\|_{2,\pi}
\]
and so, using first that a stochastic matrix has $\|L\|_\infty=1$ and then the trivial observation $\pi_* \|f\|_\infty^2 \leq \|f\|_{2,\pi}^2$,
\[
\frac{\|Lf\|}{\|f\|} \leq \frac{\|f\|_\infty + c \|f\|_{2,\pi} \sqrt{1-\nicefrac{1}{\Delta}} }{\|f\|_\infty + c \|f\|_{2,\pi}}
\leq \frac{\nicefrac{1}{\sqrt{\pi_*}} + c\left( 1-\nicefrac{1}{2\Delta} \right)}{\nicefrac{1}{\sqrt{\pi_*}} + c}.
\]

Because for every $\al \in \D$ and $i \in [I]$ we have $\alpha_i^+\geq 1$,
\begin{equation}\label{eq_pi_lower}
\pi_*= \min_i \sum_\al \frac{k(\al)}{N} \al_i^+ p_i \geq p_*
\end{equation}
and so we can conclude with $c  =\nicefrac{1}{\sqrt{p_*}}$ %(using $I\geq 2$ to say $p_*\leq \tfrac{1}{2}$)
\[
\frac{\|Lf\|}{\|f\|} \leq \frac{\nicefrac{1}{\sqrt{p_*}} + c\left( 1-\nicefrac{1}{2\Delta} \right)}{\nicefrac{1}{\sqrt{p_*}} + c}
=
1-\frac{1}{4\Delta}.
\]
%using in the final step that $c \leq \nicefrac{1}{\sqrt{p_*}}$.
\end{proof}

\begin{definition}\label{def_dynamical_variance}
A Markov chain $(X_k)_k$ on $[n]$ with $n \times n$ transition matrix $P$ and an associated function $f$ on $[n]$ has dynamical variance defined by the limit
\[
\frac{1}{t}\mathbb{V}{\rm ar}\left( \sum_{k=1}^t f\left(X_k\right) \right) 
\stackrel{\mathbb{P}}{\rightarrow}
\sigma^2_P(f),
\]
\end{definition}

This dynamical variance defines the width of the Gaussian cutoff window, and in the following result we control the change in this width.% We will 

\begin{proposition}\label{prop_sigma_approx}
Recall the dynamic \eqref{eq_layer_chain}. We find for any $q \in \sP$ when the vector of relative errors has
\[
|1-q_i/p_i| =: \epsilon_i \leq \tfrac{1}{2I} 
%\left( \leq \tfrac{1}{4}  \right)
\]
that
\[
\left| \sigma^2_{L_q}(\log q)-\sigma^2_{L_p}(\log p) \right|
\leq
11 I \Delta^3 \log^2 (I \vee 8) \|\epsilon\|_{\infty}.
\]
\end{proposition}

\begin{proof}

For $p,q \in \mathscr{P}$ write $L_p$ for the matrix \eqref{eq_layer_chain} with probabilities $p$, $(X_k)_k$ for a chain with this dynamic and for this proof write
\[
%\frac{1}{n}\mathbb{V}{\rm ar}\left( \sum_{k=1}^n \log q_{X_k} \right) 
%\stackrel{\mathbb{P}}{\rightarrow}
\psi_{p}(q)
:=
\sigma^2_{L_p}(\log q).
\]

We deal with the $p$ and $q$ in this expression one-by-one. First by Lemma \ref{lemma_dynamical_variance_bounds}
\[
\psi_{p}(q) \leq 
\psi_{p}(p) + \psi_{p}(\nicefrac{q}{p}) \leq 
\psi_{p}(p) + \sqrt{2 \Delta\mathbb{V}{\rm ar}_{\pi_p}(\nicefrac{q}{p})} \leq 
\psi_{p}(p) + \sqrt{2 \Delta}\|\epsilon\|_{2,\pi}
\]
and in the other direction we use that $\nicefrac{-\epsilon}{1-\epsilon}<1-p/q<\nicefrac{\epsilon}{1+\epsilon}$ to obtain
\[
\psi_{p}(p) \leq 
\psi_{p}(q) + \sqrt{2 \Delta}\left\| \frac{\epsilon}{1-\epsilon}\right\|_{2,\pi}\leq 
\psi_{p}(q) + \frac{4\sqrt{2 \Delta}}{3}\left\| \epsilon\right\|_{2,\pi}.
\]

So we conclude for the distance in \emph{variance}, by Lemma \ref{prop_max_variance},
\[
\begin{split}
\left|
\psi_{p}(q) - 
\psi_{p}(p)
\right|
&\leq
\frac{4\sqrt{2 \Delta}}{3}
\left\| \epsilon \right\|_{2,\pi}
\left(
\sqrt{\psi_{p}(q)} +
\sqrt{\psi_{p}(p)}
\right)
\leq
\frac{4\sqrt{2 \Delta}}{3}
\left\| \epsilon \right\|_{2,\pi}
\cdot
2\sqrt{2}\Delta \log (I\vee 8)\\
&\leq \tfrac{16}{3} \Delta^{\nicefrac{3}{2}} \|\epsilon\|_{1,\pi} 
\|\epsilon\|_{\infty}  \log (I\vee 8)
\leq \tfrac{4}{3} \Delta^{\nicefrac{3}{2}}  \|\epsilon\|_{1,\pi}  \log (I\vee 8)
\end{split}
\]

For the other distance $\left| \psi_{p}(q) - 
\psi_{q}(q) \right|$,  we introduce $(Y_k)_k$ with the dynamic $L_q$ and aim to run these chains coupled to each other. 
%Suppose both chains start at stationarity, where they can be coupled with failure probability given by the total variation distance
%\[
%\mathbb{P}(X_0 \neq Y_0)=\frac{1}{2}\sum_{i=1}^I
%\left|
%\sum_{\al \in \D} \frac{k(\al}{N} \al_i^+ (p_i-q_i)
%\right|
%\leq
%\sum_{i=1}^I
%\sum_{\al \in \D} \frac{k(\al)}{N} \al_i^+ \epsilon p_i
%=\|\epsilon\|_{1,\pi}.
%\]
Suppose that initially both chains are in the same state, this state chosen with distribution $\pi(p)$. 
From this state, we can run just the chain with the stationary dynamic $L_p$ and keep the other with dynamic $L_q$ coupled each step (from $i$ to $j$) with probability $1-\|\epsilon\|_{\infty}$, using that
%\[
%\de_{\rm TV}
%\left(
%L_q(i,\cdot),L_p(i,\cdot)
%\right)
%=\frac{1}{2}\sum_{j=1}^I
%\left|
%\frac{\sum_{\al \in \D} k(\al) \al_i^- \al_j^+ (p_j-q_j)}{\sum_{\al \in \D} k(\al) \al_i^-}
%\right|
%\leq
%\Delta\de_{\rm TV}
%\left(
%\pi(q),\pi(p)
%\right)
%\leq
%\frac{\Delta}{2}\|\epsilon\|_{1,\pi}
%\]
\[
\frac{L_q(i,j)}{L_p(i,j)}
\geq
\frac{\sum_{\al \in \D} k(\al) \al_i^- \al_j^+ p_j (1- \epsilon_j)}{\sum_{\al \in \D} k(\al) \al_i^- \al_j^+ p_j}
=1- \epsilon_j.
\]

%\JFtodo{
%can replace with chernoff via Kloeckner but a good bound on the stationary rate seems to be
%\[
%\de_{\rm TV}
%\left(
%L_q(\pi(p),\cdot),L_p(\pi(p),\cdot)
%\right)
%\leq
%\de_{\rm TV}
%\left(
%\pi(q),\pi(p)
%\right)
%\]
%so we're only improving by at most a constant factor with considerable extra mess
%}
Thus we have a number of coupled steps $\tau_1\sim \operatorname{Geom}\left(\|\epsilon\|_{\infty}\right)$ where $\tau_1$ is independent from the path of the coupled chains with dynamic $L_p$. 

After $\tau_1$, recall by \eqref{eq_separation} that we have one-step separation distance $1-\nicefrac{1}{\Delta}$, for the product chain this becomes $1-\nicefrac{1}{\Delta^2}$, and so by 
\cite[Proposition 3.2(b)]{aldous87} we can couple \emph{the product chain} to its stationary distribution every timestep. Each of these strong stationary times is also a \emph{meeting} time at a site with distribution $\pi(p)$ with probability
\[
\sum \pi(p)^2-\de_{\rm TV}
\left(
\pi(q),\pi(p)
\right)
 \geq \frac{1}{I}-\|\epsilon\|_{1,\pi}
 \geq \frac{1}{I}-\|\epsilon\|_{\infty}
 \geq \frac{1}{2I}.
\]

Thus we expect at most $2I\Delta^2$ steps until the the chains meet and we can couple them to each other. Call this meeting time $\tau_2$, and as it was constructed from strong stationary times it follows by the weak law of large numbers that we can decompose the variance
\[
\begin{split}
&\left| \psi_{p}(q) - 
\psi_{q}(q) \right|
\stackrel{\p}{\sim}
\frac{1}{n}\mathbb{V}{\rm ar}\left( \sum_{k=1}^n \log q_{Y_k} \right) 
- 
\frac{1}{n}\mathbb{V}{\rm ar}\left( \sum_{k=1}^n \log q_{X_k} \right)\\
&\stackrel{\mathbb{P}}{\rightarrow}
\frac{
\mathbb{V}{\rm ar}\left( \sum_{k=\tau_1}^{\tau_2} \log q_{Y_k}\right)
+
\mathbb{C}{\rm ov}\left( \sum_{k=0}^{\tau_1} \log q_{Y_k} ,\sum_{k=\tau_1}^{\tau_2} \log  \frac{q_{Y_k}}{q_{X_k}} \right)
-
\mathbb{V}{\rm ar}\left( \sum_{k=\tau_1}^{\tau_2} \log q_{X_k}\right)
}{\mathbb{E}(\tau_2)}\\
&\leq
\frac{
\mathbb{E}\left( \sum_{k=\tau_1}^{\tau_2} \log^2 q_{Y_k}\right)
+
\left|
\mathbb{C}{\rm ov}\left( \sum_{k=0}^{\tau_1} \log q_{Y_k} ,\sum_{k=\tau_1}^{\tau_2} \log  q_{Y_k} \right)\right|
+
\left|\mathbb{C}{\rm ov}\left( \sum_{k=0}^{\tau_1} \log q_{X_k} ,\sum_{k=\tau_1}^{\tau_2} \log  q_{X_k} \right)\right|
}{\mathbb{E}(\tau_1)}.
\end{split}
\]

For the covariance term of $Y$ we use first the law of total covariance with the strong Markov property at time $\tau_1$, and then Cauchy-Schwarz
\[
\begin{split}
\left|
\mathbb{C}{\rm ov}\left( \sum_{k=0}^{\tau_1} \log q_{Y_k} ,\sum_{k=\tau_1}^{\tau_2} \log  q_{Y_k} \right)
\right|
&=
\left|
\mathbb{C}{\rm ov}\left( 
\mathbb{E}\left( \sum_{k=0}^{\tau_1} \log q_{Y_k} \Bigg| Y_{\tau_1} \right)
,
\mathbb{E}\left( \sum_{k=\tau_1}^{\tau_2} \log  q_{Y_k} \Bigg| Y_{\tau_1} \right)
\right)
\right|
\\
&\leq
\left[
\mathbb{V}{\rm ar} \, \mathbb{E}\left( \sum_{k=0}^{\tau_1} \log q_{Y_k} \Bigg| Y_{\tau_1} \right)
\mathbb{V}{\rm ar} \, \mathbb{E}\left( \sum_{k=\tau_1}^{\tau_2} \log  q_{Y_k} \Bigg| Y_{\tau_1} \right)
\right]^{\nicefrac{1}{2}}.
%\\
%&\leq
%\left[
%\mathbb{V}{\rm ar} \, \mathbb{E}\left( \sum_{k=0}^{\tau_1} \log q_{Y_k} \Bigg| Y_{\tau_1} \right)
%\mathbb{V}{\rm ar} \left( \sum_{k=\tau_1}^{\tau_2} \log  q_{Y_k} \right)
%\right]^{\nicefrac{1}{2}}
\end{split}
\]
%and the last line could be seen as another application of the law of total covariance.

%\vspace{1em}

For {the second variance} on the shorter timescale, 
% absorb the diagonal of the product chain into a single state so that the modified product chain has $I^2-I+1$ states. Note that we already bounded the hitting time of this diagonal above, by $2I\Delta^2$.
%The row entropy of a Markov chain is the entropy of a row of its transition matrix, as in \cite{choi18}. The row entropy of a chain on $I^2-I+1$ states is maximally $H(q) \leq \log(I^2-I+1)$, and 
the row mean of the functional $\log^2 q_X+\log^2 q_Y$ is bounded by
\[
\sum_{i=1}^I \Delta p_i \log^2 q_i
+
\sum_{i=1}^I \Delta q_i \log^2 q_i
\leq
\Delta
\left(
1+\frac{1}{1-\nicefrac{1}{4}}
\right)
\sum_{i=1}^I q_i \log^2 q_i
\leq \frac{7\Delta}{3}
\log^2 (I \vee 8).
\]

We then recall that $\tau_2$ is constructed from strong stationary times for the pair, so as $Y_{\tau_1}$ can only affect the mean up to the first strong stationary time $\tau_1+X$ we can multiply by the geometric second moment
\[
\mathbb{V}{\rm ar} \, \mathbb{E}\left( \sum_{k=\tau_1}^{\tau_2} \log  q_{Y_k} \Bigg| Y_{\tau_1} \right)
=
\mathbb{V}{\rm ar} \, \mathbb{E}\left( \sum_{k=\tau_1}^{\tau_1+X} \log  q_{Y_k} \Bigg| Y_{\tau_1} \right)
\leq
2 \Delta^2 \cdot \frac{7\Delta}{3}
\log^2 (I \vee 8).
\]
%and with regard to Proposition \ref{prop_squared_log_hitting} this is larger than the bound on the row entropy. So, the expectation of this functional from arbitrary initial point up to hitting the the diagonal at most a geometric number of times with mean $2I$\JF{maybe overkill} is bounded via Proposition \ref{prop_squared_log_hitting} by
%\[
%2I \cdot
%\frac{7\Delta}{3}
%\log^2 (I \vee 8) \cdot 2I\Delta^2.
%\]
%\JFtodo{minor consequences of $2I$ factor to do}

\vspace{1em}

{The first variance} is more difficult, as a sum conditional on the terminal value requires us to look at a reversed process. However as the step and decoupling event can be sampled independently, the forwards (quasi-stationary in the sense of \cite[Section 3.6.5]{aldous-fill-2014}) dynamic is $\left(1-\nicefrac{1}{\|\epsilon\|_{\infty}}\right)L_p$ and so the quasi-stationary distribution is simply $\pi(p)$ in which the chains start.

For the time reversal $L^*_{p}$, then, write $T\sim \operatorname{Geom}(\nicefrac{1}{\Delta})$ for the first strong stationary time. This time is of course independent of the process at $T$ but in fact is further independent of the process at time $0$, and so by simulating the chain from $\tau_1$ backwards to $-\infty$
\[
\mathbb{V}{\rm ar} \, \mathbb{E}\left( \sum_{k=0}^{\tau_1} \log q_{Y_k} \Bigg| Y_{\tau_1} \right)
=
\mathbb{V}{\rm ar} \, \mathbb{E}\left( \sum_{k=(\tau_1-T)\vee 0}^{\tau_1} \log q_{Y_k} \Bigg| Y_{\tau_1} \right)
\leq \mathbb{E}\left( \left( \sum_{k=0}^{T} \log q_{Y_k}  \right)^2 \right).
\]

We can calculate $\mathbb{E}(T^2)=\Delta^2+\Delta$, 
% and writing $\hat{\pi}(p \wedge q)$ for the normalised distribution
%\[
%\mathbb{E}_{\hat{\pi}(p \wedge q)}(\log^2 q_\cdot )
%\leq 
%\frac{\mathbb{E}_{\pi(q)}(\log^2 q_\cdot )}{\sum_{z \in [I]} \sum_{\al \in \D} k(\al)   \al_z^+ ( p_z \wedge q_z )}
%\leq 
%\frac{\Delta \log^2 ( I \vee 8 )}{1-\nicefrac{1}{2I}}.
%\]
%
%Finally for the error event by Lemma \ref{lemma_dynamical_variance_bounds}
%\[
%\frac{
%\mathbb{V}{\rm ar}\left( \sum_{k=0}^{\tau_2} \log q_{Y_k}\right)
%}{\mathbb{E}(\tau_2)}
%\leq
%\lim_{n \rightarrow \infty}
%\frac{1}{n}\mathbb{V}{\rm ar}\left( \sum_{k=1}^n \log q_{Y_k} \right) 
%\leq 
%2 \Delta \mathbb{V}{\rm ar}_{\pi_q}\left(
%\log q
%\right)
%\leq
%2 \Delta
%\log^2(I \vee 8).
%\]
so altogether we bound 
\[
\begin{split}
\left|
\mathbb{C}{\rm ov}\left( \sum_{k=0}^{\tau_1} \log q_{Y_k} ,\sum_{k=\tau_1}^{\tau_2} \log  q_{Y_k} \right)
\right|
&\leq 
\sqrt{ 
 (\Delta^2+\Delta)\Delta \log^2 ( I \vee 8 )
\cdot
 \frac{14I\Delta^3}{3}
\log^2 (I \vee 8) }\\
&\leq
\Delta^3 \log^2 (I \vee 8)\sqrt{7I}
 \end{split}
\]
and we have the same bound by the same argument for the covariance expression of $X$.
%\[
%+2
%\mathbb{E}\left(
%\left(
%1-\frac{1}{\Delta}
%\right)^{\nicefrac{\tau_1}{2}}
%\right)
%\cdot
%2 \Delta
%\log^2(I \vee 8)
%\cdot
%\left(
%2I\Delta^2+
%\mathbb{E}(\tau_1)
%\right)
%%\left(
%%\tfrac{2}{\Delta \|\epsilon\|_{1,\pi}}
%%+2I\Delta^2
%%\right)
%\]

%Summing up the covariance bounds, we have:
%\[
%\begin{split}
%\left|
%\mathbb{C}{\rm ov}\left( \sum_{k=0}^{\tau_1} \log q_{Y_k} ,\sum_{k=\tau_1}^{\tau_2} \log  q_{Y_k} \right)
%\right|
%&\leq
%\sqrt{
%2 \Delta
%\log^2(I \vee 8)
%\cdot
%\frac{\Delta+1}{ \|\epsilon\|_{1,\pi}}
% \cdot
%\frac{14 I \Delta^3}{3}
%\log^2 (I \vee 8)
%}
%+
%\frac{14 I \Delta^3}{3}
%\log^2 (I \vee 8)\\
%&\leq
%\frac{8 I \Delta^3\log^2 (I \vee 8)}{\sqrt{\|\epsilon\|_{1,\pi}}};
%\end{split}
%\]
%\[
%\begin{split}
%\left|
%\mathbb{C}{\rm ov}\left( \sum_{k=0}^{\tau_1} \log q_{X_k} ,\sum_{k=\tau_1}^{\tau_2} \log  q_{X_k} \right)
%\right|
%&\leq
%\sqrt{
%\frac{8}{3} \Delta
%\log^2(I \vee 8)
%\cdot
%\frac{\Delta+1}{ \|\epsilon\|_{1,\pi}}
%\cdot
%\frac{14 I \Delta^3}{3}
%\log^2 (I \vee 8)
%}
%+
%\frac{14 I \Delta^3}{3}
%\log^2 (I \vee 8)\\
%&\leq
%\frac{10 I \Delta^3\log^2 (I \vee 8)}{\sqrt{\|\epsilon\|_{1,\pi}}}.
%\end{split}
%\]

So we find for this second distance
\[
\left| \psi_{p}(q) - 
\psi_{q}(q) \right|
\leq
\frac{
\frac{14 I \Delta^3}{3}
\log^2 (I \vee 8)
+
2\Delta^3 \log^2 (I \vee 8)\sqrt{7I}
}{\mathbb{E}(\tau_1)}
\leq
10 I \Delta^3 \log^2 (I \vee 8) \|\epsilon\|_{\infty}
\]
and then conclude
\[
\left| \psi_{p}(p) - 
\psi_{q}(q) \right|
\leq
\tfrac{4}{3} \Delta^{\nicefrac{3}{2}}  \|\epsilon\|_{1,\pi}  \log (I\vee 8)
+
10 I \Delta^3 \log^2 (I \vee 8) \|\epsilon\|_{\infty}
\]
which gives the result as an upper bound.
% and using $\Delta\sigma^2_{L}\geq \mathbb{V}{\rm ar}_\pi \log p_i \geq B \frac{\log^2 \log N}{\log N}$
%\[
%%\left| \sigma_{L_q}(q)-\sigma_{L_p}(p) \right|\leq
%\frac{\epsilon}{1-\epsilon}
%+
%\frac{\epsilon \Delta \log^2 \log N}{18(\sigma_{L_q}(q)+\sigma_{L_p}(p))} 
%\leq
%\frac{\epsilon}{1-\epsilon}
%+
%\frac{\epsilon \Delta  \log \log N}{36} 
%\sqrt{\frac{\Delta \log N}{B}}.
%\]
\end{proof}

\begin{definition}[$Q_{v,t}$]\label{def_Q}
We define the set of paths of length $t$ between two vertices as you would expect, but note that the path also contains the layer that sourced each edge:
\[
\mathfrak{P}^t_{v_0,v_t}=\left\{
\mathfrak{p}=
\left(
(v_0,i_1),\dots,(v_{t-1},i_{t})
\right): \, 
\forall k \in [t], \quad
A^{(i_{k})}_{v_{k-1} v_{k}}>0
\right\}.
\]

Then the weight of the path (which should be thought of as just the probability to take the path) can be read from the layer labels
\[
\mathbf{w}\big(\left(
(v_0,i_1),\dots, (v_{t-1},i_t)
\right)\big)
=
\prod_{k=1}^{t} p_{i_k}
\]
which gives the network version of the path weight distribution from $v$ in time $t$
\[
Q_{v,t}(\varphi)=
\sum_{w=1}^N
\sum_{\mathfrak{p}\in \mathfrak{P}^t_{v,w}}
\mathbf{w}(\mathfrak{p})
\mathbbm{1}_{\mathbf{w}(\mathfrak{p})>\varphi}.
\]
\end{definition}

For a small constant multiple of the logarithm
\begin{equation}\label{eq_h_def}
h=
\left\lfloor
\frac{\log N}{10 \log \Delta+10 \log I}
\right\rfloor
\end{equation}
we introduce a set $V$ of \emph{nice vertices}
\begin{equation}\label{eq_nice_vertex}
V:=\{
v \in [N] : B^+(v,h) \text{ is a tree}.
\}
\end{equation}

\begin{lemma}\label{lemma_rerooting}
With probability at least $1-N^{\nicefrac{-1}{6}}$ 
we have
\[
\forall p \in \mathscr{P}, 
\forall i \in [N], \, \forall t \in \mathbb{N} \qquad
P^t(i,[N]\setminus V)
\leq
\left(p^*\right)^{t \wedge h}.
\]
\end{lemma}

\begin{proof}
First we show that, with high probability and for every vertex, the forward ball of radius $2h$ has tree surplus at most $1$.
This is very similar to \cite[Section 3.2]{bordenave18} except that we can only select heads within a particular layer. 

The maximal ball size is $\nicefrac{(\Delta I)^{2h+1}}{(\Delta I-1)}$ and so because each exploration removes one available head, the $k$\textsuperscript{th} edge explored in the ball chooses a previously contacted vertex with probability bounded by
\[
\frac{\Delta k}{N - \nicefrac{(\Delta I)^{2h+1}}{(\Delta I-1)}}\leq 2 \Delta N^{-\nicefrac{4}{5}}.
\]

These exploration events are independently sampled and so we have a binomial.
We can bound this binomial by a Poisson then with mean 
\[
N^{\nicefrac{1}{5}}\cdot 
\frac{2 \Delta N^{-\nicefrac{4}{5}}}{1-2 \Delta N^{-\nicefrac{4}{5}}}
%=N^{\nicefrac{1}{5}}\cdot 2 \Delta N^{-\nicefrac{4}{5}}
\leq 3 \Delta N^{-\nicefrac{3}{5}}
\]
 which then is at least $2$ with probability 
\[
1-e^{-3 \Delta N^{-\nicefrac{3}{5}}}(1+3 \Delta N^{-\nicefrac{3}{5}})
\leq
1-(1-3 \Delta N^{-\nicefrac{3}{5}})(1+3 \Delta N^{-\nicefrac{3}{5}})
\leq
9 \Delta^2 N^{-\nicefrac{6}{5}}
.
\]

Therefore over $[N]$, we expect fewer than $9 \Delta^2N^{-\nicefrac{1}{5}}$ vertices with two surplus in their forward ball and by Markov's inequality see none with error probability at most $9 \Delta^2N^{-\nicefrac{1}{5}}$. Then, because the maximal probability we step towards the surplus edge is $p^*$ at every timestep,
\[
\forall i \in [N], \, \forall t \in \mathbb{N} \qquad
P^t(i,[N]\setminus V)
\leq
(p^*)^{t \wedge h}.
\]
%from which the conclusion follows as $|p_i-{\rm r}(p)_i|\leq \log^{-3}N$.
\end{proof}

We approximate the typical path weight from some type $\al \in \D$ using the layer chain $(M_k)_{k \in \mathbb{N}^+}$, with initial condition
\[
\p(M_1=i)=\al^+_i p_i
\]
and transition matrix
\[
L_{ij}=
\frac{\sum_{\al \in \D} k(\al) \al_i^- \al_j^+ p_j}{\sum_{\al \in \D} k(\al) \al_i^-}
\]
to define
\begin{equation}\label{eq_q_def}
q^{(\alpha)}_t(\varphi)=
\mathbb{P}\left(
\prod_{k=1}^{t} 
p_{M_k}
>\varphi
\right).
\end{equation}

On the network, this approximates an averaged version of the local path weight $Q$ from Definition \ref{def_Q}.

\begin{definition}[$\bar{Q}_{v,t}$]
We use a small time period after which the path weights are guaranteed to at least be small (recall $p^*=\max_{i\in[I]}p_i$)
\[
\ell=
\left\lfloor
\frac{3 \log \log N}{\log \nicefrac{1}{p^*}}
\right\rfloor
\]
to construct the averaged version of ${Q}_{v,t}$
\[
\bar{Q}_{v,t}(\varphi)=
\sum_{w=1}^N
\left[P^{\ell}\right]_{v\rightarrow w}
Q_{w,t}(\varphi).
\]
\end{definition}

This approximation is in the following sense.

\begin{proposition}\label{prop_Qbar_and_q}
For $t \leq \log^{\nicefrac{4}{3}} N$, large $N$, and $\varphi$ also a function of $N$,
\[
q^{(\tau(v))}_t\left(\left(1+\frac{1}{\log N}\right)^2\varphi\right)-
\frac{2}{\sqrt{\log N}}
\leq
\bar{Q}_{v,t}(\varphi)
\leq
q^{(\tau(v))}_t\left(\left(1-\frac{1}{\log N}\right)^2\varphi\right)+
\frac{2}{\sqrt{\log N}}
\]
simultaneously for every $p \in \mathscr{P}$ and $v \in V$, 
with probability at least $1-e^{-\sqrt[3]{\log N}}$.
\end{proposition}

%\JFtodo{improving the error probability in Theorem \ref{thm_cleaner_simultaneous_cutoff} starts with the orders here
%
%however I'm not sure that we want to be writing ``with prob $1- e^{-\log^{1-\epsilon}N}$'' anyway}

\begin{proof}
We follow \cite[Lemma 9]{bordenave18}. Take some $v \in V$ with type $\al \in \D$, and explore the network using $\lfloor \log^2 N \rfloor$ walkers from $(v,i)$, each running until time $\ell + t$.

Let $Z_{vi}$ be the event that none of these walkers find a cycle before time $\ell$ (which is guaranteed if $v \in V$) and also that none of the edge paths explored between distance $\ell$ and $\ell + t$ have weight less than or equal to $\varphi$. Because (recall $A$ is the adjacency matrix of our multitype configuration model)
\[
\mathbb{P}\left(
Z_{vi} \big| A
\right)
\geq \mathbbm{1}_{v \in V} \left(
\bar{Q}_{v,t}(\varphi)
\right)^{\lfloor \log^2 N \rfloor},
\]
by Markov's inequality we find
\[
\begin{split}
\mathbb{P}\left(
Z_{vi} 
\right)
&\geq
\mathbb{E}\left( \left(
\bar{Q}_{v,t}(\varphi)
\right)^{\lfloor \log^2 N \rfloor}
; v \in V
\right)\\
&\geq
(q^{(\tau(v))}_t(\varphi)+\epsilon)^{\lfloor \log^2 N \rfloor}
\mathbb{P}\left( 
\bar{Q}_{v,t}(\varphi)
\geq q^{(\tau(v))}_t(\varphi)+\epsilon
, v \in V
\right).
\end{split}
\]

We can also upper bound $\mathbb{P}\left(
Z_{vi} 
\right)$ by exploring the network with the walkers. For this, iteratively for each walker:
\begin{itemize}
\item Each previous path had weight from distance $0$ to $\ell$ at most $(p^*)^{\ell}$, and so the probability to be at a previously explored vertex at time $\ell$ is at most
\[
\lfloor \log^2 N \rfloor (p^*)^{\ell} \leq \frac{1}{\log N}
\]
\item At each of the $t \cdot \lfloor \log^2 N \rfloor \leq \log^{\nicefrac{10}{3}} N $ vertex explorations beyond the $\ell$-ball we have at most $ \nicefrac{\Delta \log^{\nicefrac{10}{3}} N}{N} $ probability to discover a cycle by repeating a vertex, so this happens with a probability of smaller order.
\end{itemize}

This coupling gives 
$
\mathbb{P}\left(
Z_{vi} 
\right)
\leq
(q^{(\tau(v))}_t(\varphi)+\nicefrac{(1+o(1))}{\log N})^{\lfloor \log^2 N \rfloor}
$
 and we conclude as in \cite[Lemma 9]{bordenave18}, using the union bound and Markov's inequality
 \[
 \begin{split}
 \mathbb{P}&\left(
 \exists v \in V \, : \, 
 \bar{Q}_{v,t}(\varphi) \geq q^{(\tau(v))}_t(\varphi)
 +
\frac{1}{\sqrt{\log N}}
 \right)\\
 &\leq
 N \left(
 \frac{q^{(\tau(v))}_t(\varphi)
 +
\nicefrac{1}{\log N}
+o\left(
\nicefrac{1}{\log N}
\right)
}{q^{(\tau(v))}_t(\varphi)
 +
\nicefrac{1}{\sqrt{\log N}}}
 \right)^{\lfloor \log^2 N \rfloor}
 \leq
 N \left(
 1-\nicefrac{1}{\sqrt{\log N}}
  \right)^{\lfloor \log^2 N \rfloor}
  \sim e^{-\sqrt{\log N}}.
  \end{split}
 \]
 with symmetrical arguments for the reverse direction.
 
 So then, in both directions, by the union bound we have this bound at every representative ${ r}(p)$ in the grid $G$ with failure probability $\lesssim 2e^{-\sqrt{\log N}}\log^{3I}N \leq e^{-\sqrt[3]{\log N}}$. Therefore, we can apply Proposition \ref{prop_Qq_approx} twice in each direction to obtain
\[
q^{(\tau(v))}_t\left(\varphi\left(1+\frac{1}{\log N}\right)^2\right)-
\frac{1}{\sqrt{\log N}}-\frac{2}{\log N}
\leq
\bar{Q}_{v,t}(\varphi)
\]
\[
\leq
q^{(\tau(v))}_t\left(\varphi\left(1-\frac{1}{\log N}\right)^2\right)+
\frac{1}{\sqrt{\log N}}+\frac{2}{\log N}
\]
which gives the claimed bounds for large $N$.
\end{proof}

%\clearpage
\section{Cutoff}\label{sec_cutoff}

We first present the following lemma to be able to relate the dynamical variance to the stationary variance, the upper bound of which is well-known but the lower bound we could not find in the literature.

\begin{lemma}\label{lemma_dynamical_variance_bounds}
A Markov chain with separation distance $s(1)\leq 1-\nicefrac{1}{\Delta}$  and dynamical variance $\sigma^2$ (see Definition \ref{def_dynamical_variance})
%\[
%\frac{1}{n}\mathbb{V}{\rm ar}\left( \sum_{t=1}^n f(X_t) \right) 
%\stackrel{\mathbb{P}}{\rightarrow}
%\sigma^2_L
%\]
has
\[\frac{1}{\Delta^2}
\mathbb{V}{\rm ar}_\pi\left(
f
\right)
\leq \sigma^2
\leq
(2 \Delta-1) \mathbb{V}{\rm ar}_\pi\left(
f
\right).
\]
\end{lemma}

\begin{proof}
The second inequality relies on the expression
\[
\sigma^2=\mathbb{V}{\rm ar}_\pi\left(
f
\right)
+2\sum_{t=1}^\infty \mathbb{C}{\rm ov}\left(
f(X_0),f(X_t)
\right)
\]
in which we put Lemma \ref{lemma_covariance_and_separation} with $s(t) \leq \left( 1-\nicefrac{1}{\Delta} \right)^t$. For the first inequality, consider the sequence of strong stationary times $(S_k)_k$ with $S_0=1$ and
\[
\p\left(
S_j-S_{j-1}=t
\right)
=\frac{1}{\Delta}\left(
1-\frac{1}{\Delta}
\right)^{t-1} \qquad \text{on } t \geq 1.
\]

%\[
%\mathbb{V}{\rm ar}\left( \sum_{t=1}^{S_n-1} f(X_t)\right) 
%=
%\mathbb{E}
%\mathbb{V}{\rm ar}\left( \sum_{t=1}^{S_n-1} f(X_t) \Big| (S) \right) 
%+
%\pi(f)^2
%\mathbb{V}{\rm ar}(S_n)
%\]
%and also
By the independence given at strong stationary times
\[
\mathbb{V}{\rm ar}\left( \sum_{t=1}^{n} f(X_t) \Bigg| (S_k)_k \right) 
=
 \sum_{j=1}^{\infty}
\mathbb{V}{\rm ar}\left( \sum_{t=S_{j-1}}^{(S_j-1) \wedge n} f(X_t) \Bigg| (S_k)_k \right).
\]

Now, note for one of these terms
\[
\begin{split}
\mathbb{E}\left(
\mathbb{V}{\rm ar}\left( \sum_{t=S_{j-1}}^{(S_j-1) \wedge n} f(X_t) \Bigg| (S_k)_k \right)\right)
&=
\mathbb{E}\left(
\mathbb{V}{\rm ar}\left( \sum_{t=S_{j-1}}^{(S_j-1) \wedge n} f(X_t) \Bigg| (S_k)_k \right) \mathbbm{1}_{\{S_{k-1}<n\}} \right)\\
&\geq
\mathbb{E}\left(
\mathbb{V}{\rm ar}\left( \sum_{t=S_{j-1}}^{S_j-1} f(X_t) \Bigg| (S_k)_k \right)
\mathbbm{1}_{\{S_j-S_{j-1}=1,S_{j-1} \leq n\}}
\right)\\
&=
\mathbb{E}\left(
\mathbb{V}{\rm ar}_\pi\left( f \right)
\mathbbm{1}_{\{S_j-S_{j-1}=1,S_{j-1} \leq n\}}
\right)\\
\end{split}
\]
so by the law of total variance we have
\[
\begin{split}
\mathbb{V}{\rm ar}\left( \sum_{t=1}^{n} f(X_t)  \right)
&=
\mathbb{E}
\mathbb{V}{\rm ar}\left( \sum_{t=1}^{n} f(X_t) \Bigg| (S_k)_k \right) \\
&\geq
\mathbb{V}{\rm ar}_\pi\left( f \right)
\sum_{j=1}^\infty
\mathbb{P}\left(
S_j-S_{j-1}=1,S_{j-1} \leq n
\right)\\
&=\frac{1}{\Delta}\mathbb{V}{\rm ar}_\pi\left( f \right)\sum_{j=1}^\infty
\mathbb{P}\left(
S_{j-1} \leq n
\right)
\sim \frac{n}{\Delta^2}\mathbb{V}{\rm ar}_\pi\left( f \right).
\end{split}
\]
where the final asymptotic follows from the elementary renewal theorem.
\end{proof}

The following Berry-Esseen theorem allows us to adapt the i.i.d. edge control in \cite{bordenave18} to our context where these edge probabilities instead form a Markov chain. This result is an application of \cite[Theorem C]{kloeckner2019} to get the cutoff shape in terms of the dynamical variance 
$
\sigma^2
$ (see Definition \ref{def_dynamical_variance}). We check Assumptions \ref{assumptions_kloeckner} in Lemma \ref{lemma_banach}.

\begin{definition}[Uniform metric]
Two random variables $X, Y$ on $\mathbb{R}$ with distribution functions $F_X, F_Y$ have uniform distance
\[
K(X,Y):=
\sup_{c \in \mathbb{R}}\left|
F_X\left(
c
\right)
-F_Y\left(
c
\right)
\right|.
\]
\end{definition}

\begin{assumption} \label{assumptions_kloeckner}
The norm $\|\cdot\|$ has the following properties $\forall f, g$:
\begin{itemize}
\item $\|f\|\geq\|f\|_\infty$;
\item $\|f\|\|g\|\leq\|f\|\|g\|$;
\item $\|1\|=1$.
\end{itemize}
\end{assumption}

\begin{theorem}[ {\cite[Theorem C]{kloeckner2019}} ]
\label{thm_clt}
Consider a chain $(X_k)_k$ with transition matrix $P$, stationary distribution $\pi$ and arbitrary initial condition. Assume the chain is a contraction in the norm with Assumptions \ref{assumptions_kloeckner} (for example the norm \eqref{eq_norm_def}) with parameter $\delta$, that is
\[
\pi(x)=0 \quad \implies \quad
\frac{\|Px\|}{\|x\|} \leq 1-\delta.
\]

Then the empirical sum $S_t=\sum_{k=1}^t f(X_k)$, with mean term $\mu=\pi(f)$ and dynamical variance 
$
\sigma^2
$, 
has distance to its central limit bounded by
\[
K\left(
S_t, \mu t + Z \sigma \sqrt{t}
\right)
\leq
\frac{148}{\sqrt{t}}
\left(
1+\frac{1}{\delta}
\right)^2
\left(
1
+
\frac{\left\|f-\mu\right\|}{\sigma}
\right)^3
\]
where $Z$ denotes a standard Gaussian.
\end{theorem}

\begin{proposition}\label{prop_clt_application}
For every $\al \in \D$ we find
\[
\left|
q^{(\alpha)}_t\left(
\varphi\right) - \Phi\left(
\frac{-\mu t - \log \varphi}{\sigma_p \sqrt{t}} 
\right)
\right|
\leq
\frac{(10 \Delta)^5}{\sqrt{p^3_* \, t}}
.
\]
\end{proposition}

\begin{proof}
Define
\[
S_t=\sum_{k=1}^t { f}(M_k),
\qquad
{ f}(i)=
-\log \left(
p_i
\right),
\]
recall
\[
q^{(\alpha)}_t\left(
\varphi\right)
=
\mathbb{P}_{\alpha}\left(
S_t
< -\log \varphi
\right),
\qquad
\mu =
-\sum_{i=1}^I \sum_{\al \in \D} \frac{k(\al)}{N} \al_i^+ p_i \log p_i,
\]
and write the abbreviation
\[
\sigma_p=\sigma_{L_p}(\log p).
\]

%and we check in Lemma \ref{lemma_banach} that the norm of \eqref{eq_norm_def} does form a Banach algebra on $\mathbb{R}^n$, which allows us to say the norm is permitted for applying the general results of \cite{kloeckner2019}. 
We find by Theorem \ref{thm_clt}
\[
\left|
q^{(\alpha)}_t\left(
\varphi\right) - \Phi\left(
\frac{-\mu t - \log \varphi}{\sigma_p \sqrt{t}} 
\right)
\right|
\leq
\frac{148}{\sqrt{t}}
\left(
1+\frac{1}{\delta}
\right)^2
\left(
1
+
\frac{\left\|f-\mu\right\|}{\sigma_p}
\right)^3.
\]

The constants are as follows:
\begin{itemize}
\item from Corollary \ref{cor_contraction} we can take $\delta=\frac{1}{4\Delta}$;
%\item In \eqref{eq_separation} we found that the layer chain has separation distance $s(1)\leq 1-1/\Delta$, by which we can apply \cite[Proposition 1.10]{fill90} to construct a new strong stationary time every timestep with probability $1/\Delta$. Hence
%\begin{equation}\label{eq_ergodic_variance}
%\sigma_L^2=
%\lim_{n \rightarrow \infty}\frac{1}{n}\mathbb{V}{\rm ar}\left( \sum_{k=1}^n X_k \right)
%\geq \frac{1}{\Delta}\operatorname{Var}_\pi(f) 
%\end{equation}
\item We lower bound the dynamical variance in Lemma \ref{lemma_dynamical_variance_bounds}. By also recalling the norm definition in \eqref{eq_norm_def}, the lower bound \eqref{eq_pi_lower}, and that $\pi_* \left\|f-\mu\right\|_\infty^2 \leq \operatorname{Var}_\pi(f)$,
\[
\frac{\left\|f-\mu\right\|}{\sigma_p}
\leq
\frac{\|f-\mu\|_\infty + c\sqrt{\operatorname{Var}_\pi(f) }}{\tfrac{1}{\Delta}\sqrt{\operatorname{Var}_\pi(f) }}
\leq
\frac{\Delta}{\sqrt{\pi_*}}+\Delta c
\leq
\frac{\Delta}{\sqrt{p_*}}+\Delta c.
\]
%\JFtodo{not sure this should be more than logarithmic in $p_*$}
\end{itemize}

So, inserting both bounds with $c=\nicefrac{1}{\sqrt{p_*}}$,
\[
\left|
q^{(\alpha)}_t\left(
\varphi\right) - \Phi\left(
\frac{-\mu t - \log \varphi}{\sigma_p \sqrt{t}} 
\right)
\right|
\leq
\frac{148}{\sqrt{t}}
\left(
1+4\Delta
\right)^2
\left(
1
+
\frac{2 \Delta}{\sqrt{p_*}}
\right)^3
\]
leading to the stated bound with constant factor $148 \cdot (5 \Delta)^2 \cdot (3\Delta)^3=99900\Delta^5$.
\end{proof}

Given this cutoff profile for the approximating Markov chain, we immediately derive a result for the network.

\begin{proposition}\label{prop_network_path_weights}
For $t \sim t_p$, write
\[
\lambda_{\varphi}=
\frac{\mu_p t + \log \varphi}{\sigma_p \sqrt{t}}.
\]

Take $N$ large enough and $\theta$ sufficiently small. Then we have, for $A \in [1, \theta \log^{\nicefrac{1}{3}}N]$,
\[
\sup_{p \in \Pab}
\max_{i \in [N]}
\left|
Q_{i,t}(\varphi)
- \Phi\left(
-\lambda_{\varphi}
\right)
\right|
\lesssim
\frac{10^5 \Delta^2  \sqrt{\log \Delta I}}{A^{\nicefrac{3}{2}}}
+\frac{2 \Delta \left( \log \log N \right)^2 \sqrt{\log \Delta I}}{ \log^{\nicefrac{1}{3}} N},
\]
with probability at least $1-2e^{-\sqrt[3]{\log N}}$.
\end{proposition}

\begin{proof}
$\bar{Q}$ at time $t-\ell$ depends on edges that $Q$ at time $t$ also depends on.
Take $\ell$ steps to hit $V$ %and mix the layer chain
 (recalling Lemma \ref{lemma_rerooting}) and another $\ell$ before $\bar{Q}$ starts recording. This gives
\[
\max_{v \in [N]}
Q_{v,t}(\varphi)
\leq
\max_{v \in V}
\bar{Q}_{v,t-2\ell}(\varphi )
+
(p^*)^{\ell}
\]
%(recall \eqref{eq_p_bounds}) 
and also, because we can control the minimum edge weight
\[
\bar{Q}_{v,t-2\ell}(\varphi )
%\leq
%\bar{Q}_{v,t}\left(
%\varphi
%\left(
%\frac{ p_* }{p^*}
% \right)^{2\ell}
%\right)
\leq
\bar{Q}_{v,t}\left(
\varphi
\left(
p_* 
 \right)^{2\ell}
\right).
\]

The idea behind this bound is to use
\[
\left|\log\left( p_*^{2\ell} \right)\right|
\lesssim
\frac{6 \log \log N \log \frac{1}{p_*}}{\log 2}
\leq
3  \log^2 \log N 
,
\]
then lower bound $t_p$ by Lemma \ref{prop_max_entropy} and $\sigma_p$ by Lemmas \ref{prop_min_general_variance} and \ref{lemma_dynamical_variance_bounds}
\[
\frac{|\log\left( p_*^{2\ell} \right)|}{\sigma_p\sqrt{t}}
\sim
\frac{|\log\left( p_*^{2\ell} \right)|}{\sigma_p\sqrt{t_p}}
\lesssim
\frac{3 \left( \log \log N \right)^2 \sqrt{\log \Delta I}}{\sigma_p\sqrt{\log N}}
\leq
\frac{3 \Delta \left( \log \log N \right)^2 \sqrt{\log \Delta I}}{\sqrt{A} \log^{\nicefrac{1}{3}} N}
\]
i.e. this polylogarithmic change in $\varphi$ is a small correction to $\lambda_{\varphi}$, given that $A \geq 1$.

By Propositions \ref{prop_Qbar_and_q} and \ref{prop_clt_application} and that $\Phi$ has Lipschitz constant $\nicefrac{1}{\sqrt{2 \pi}}<\nicefrac{2}{3}$,
\[
\begin{split}
\max_{v \in [N]}
Q_{v,t}(\varphi)
&\leq
\max_{\al \in \D}
q^{(\alpha)}_t\left(
\varphi
\left(
 p_* 
 \right)^{2\ell}
 \left(
 1
 -
\frac{1}{\log N}
\right)^2
\right)
+
(p^*)^{\ell}
+
\frac{2}{\sqrt{\log N}}\\
&\lesssim
\Phi\left(-\lambda_{\varphi}+\frac{3 \Delta \left( \log \log N \right)^2 \sqrt{\log \Delta I}}{\sqrt{A} \log^{\nicefrac{1}{3}} N}
-\frac{2}{\sigma_p \sqrt{t}}
 \log  \left(
 1
 -
\frac{1}{\log N}
\right)
\right)\\&\hspace{15em}+\frac{(10 \Delta)^5}{\sqrt{p^3_* \, t}}
+
\frac{1}{\log^3 N}
+
\frac{2}{\sqrt{\log N}}\\
&\leq
\Phi\left(-\lambda_{\varphi}+\frac{3 \Delta \left( \log \log N \right)^2 \sqrt{\log \Delta I}}{\sqrt{A} \log^{\nicefrac{1}{3}} N}
\right)+\frac{(10 \Delta)^5}{\sqrt{p^3_* \, t}}
+
\frac{1}{\log^3 N}
+
\frac{2}{\sqrt{\log N}}\\
&\leq
\Phi(-\lambda_{\varphi})
+\frac{2 \Delta \left( \log \log N \right)^2 \sqrt{\log \Delta I}}{\sqrt{A} \log^{\nicefrac{1}{3}} N}
+\frac{(10 \Delta)^5}{\sqrt{\nicefrac{A^3}{\log \Delta I}}},
\end{split}
\]
while in the other direction excluding $2 \ell$ steps has a controllable effect on the weight of the path
\[
\min_{v \in [N]}
Q_{v,t}(\varphi)
\geq
\min_{v \in V}
\bar{Q}_{v,t-2\ell}(\varphi (p_*)^{-2\ell})
-
(p^*)^{\ell}
\geq
\min_{v \in V}
\bar{Q}_{v,t}(\varphi (p_*)^{-2\ell})
-
(p^*)^{\ell}
\]
and so
\[
\begin{split}
\min_{v \in [N]}
Q_{v,t}(\varphi)
&\geq
\min_{\al \in \D}
q^{(\alpha)}_t\left(
\varphi
\left(
p_*
 \right)^{-2\ell}
 \left(
 1
+
\frac{1}{\log N}
\right)^2
\right)
-\frac{1}{\log^3 N}
-
\frac{2}{\sqrt{\log N}}\\
&\gtrsim
\Phi(-\lambda_{\varphi})
-\frac{2 \Delta \left( \log \log N \right)^2 \sqrt{\log \Delta I}}{\sqrt{A} \log^{\nicefrac{1}{3}} N}
-
\frac{2}{\sigma_p \sqrt{6 \pi t_p}}
 \log  \left(
 1
 +
\frac{1}{\log N}
\right)
-\frac{(10 \Delta)^5}{\sqrt{p^3_* \, t_p}}\\
&\gtrsim
\Phi(-\lambda_{\varphi})
-\frac{2 \Delta \left( \log \log N \right)^2 \sqrt{\log \Delta I}}{\sqrt{A} \log^{\nicefrac{1}{3}} N}
-
\frac{2 \sqrt{\log(\Delta I)}}{ \sqrt{6 \pi A}\log^{\nicefrac{4}{3}}N}
-\frac{(10 \Delta)^5}{\sqrt{\nicefrac{A^3}{\log \Delta I}}}
\end{split}
\]
%\JFtodo{
%maybe a bound like
%\[
%\begin{cases}
%\frac{2 \cdot 10^5 \Delta^2 \sqrt{\log \Delta I}}{\sqrt{A^3}} & A \leq \frac{10^5 \Delta^2 \log^{\nicefrac{1}{3}}N}{\log^2 \log N}\\
%\frac{2 \left( \log \log N \right)^2 \sqrt{\log \Delta I}}{ \log^{\nicefrac{1}{3}} N} & \text{otherwise}
%\end{cases}
%\]
%}

Error probabilities from Lemma \ref{lemma_rerooting} and Proposition  \ref{prop_Qbar_and_q} were $4 \Delta^2N^{-1/5}$ and $e^{-\sqrt[3]{\log N}}$ which gives the stated error probability by summation.
\end{proof}

\begin{definition}[The trees $T_x$ and tail sets $\mathcal{E}_x^{(i)}$]\label{def_T_construction}
Construct at every root $x\in[N]$ the $x$-rooted tree $T_x$  by the following algorithm, throughout which we associate to every vertex $y \in T_v$ the weight $w(y)$ of the path from the root to its parent edge. Further, we omit explored edges producing a cycle (by iteratively adding only the \emph{undiscovered} children) so that this path is unique.
\begin{description}
\item[Stage 1] Initially $T_x=(\{x\},\emptyset)$ with $x$ set as the root.
\item[Stage 2] Substitute $p$ with its representative ${ r}(p)$, so that for every grid element $g \in G$ (as described in Section \ref{sec_approximation}) we have only one family of trees $(T_x)_{x \in [N]}$.
\item[Stage 3\label{stage_3}] Explore the vertex of maximum weight (or the minimum index of those in the case of ties), adding any new children to the tree and omitting cycle edges, until the maximal weight in the unexplored vertices of $T_x$ is at most $\tfrac{1}{p_* N}\log^{4+4\Delta \log (\Delta I)} N$.
%\item Then, explore the children of every unexplored vertex one additional time.
\item[Stage 4\label{stage_4}] Now in the breadth-first order, iteratively explore a vertex of type
\[
j=\argmax_i (W_i-\pi(i))
\]
where $W_i$ denotes the total weight of unexplored vertices with parent edge of type $i$. This continues until 
\[
\de_{\rm TV}(W,\pi_{{ r}(p)})\leq \frac{1}{\log^2 N}.
\]
or until some vertex is explored to an additional depth $4\Delta \log \log N$, whichever comes first.
\end{description}
Write $\mathcal{E}_x$ for the set of unexplored tails (out-edges) at the end of the $T_x$ construction. 
List these as sets $\mathcal{E}_x^{(1)}, \dots, \mathcal{E}_x^{(I)}$ containing the tails of each type. Here (as everywhere else in the article) we take $N$ large enough.
\end{definition}

\begin{proposition}\label{prop_cycle_weight}
Almost surely for $A \geq 1$
\[
\max_{x \in [N]}
\sup_{p \in \Pab}
|T_x| \leq \frac{2 N}{ \log^{\nicefrac{8}{3}} N},
\]
and the total weight $C_x$ of cycle edges omitted in the construction of $T_x$ has
\[
\sup_{p \in \Pab}
\max_{v \in V} C_v \lesssim
\frac{2}{\log N},
\]
with probability $1-e^{-\log^2 N}$
\end{proposition}

\begin{proof}
No vertex explored in \ref{stage_3} of the construction had weight less than $\nicefrac{\log^{4+4\Delta \log (\Delta I)} N}{p_* N}$ and so the weight of unexplored vertices is at least $\nicefrac{\log^{4+4\Delta \log (\Delta I)} N}{ N}$. These weights form a sub-probability distribution on the set of unexplored vertices after \ref{stage_3}, and so there are at most $\nicefrac{ N}{\log^{4+4\Delta \log (\Delta I)} N}$ such vertices. After \ref{stage_4}, due to the limit on the additional depth, there are instead at most
\[
\frac{ N}{\log^{4+4\Delta \log (\Delta I)} N}
\cdot
(\Delta I)^{4\Delta \log \log N}
=
\frac{  N}{\log^4 N}.
\]

This just counts the final layer, but every vertex in $T_x$ is on some path to a vertex in this final layer and by construction (and using that $ p^* \leq 1-p_*$) the maximum depth $D$ of the tree has
\[
D\leq \frac{\log \left(
\nicefrac{p_* N}{\log^{4+4\Delta \log (\Delta I)} N}
\right)}{\log \nicefrac{1}{p^*}}
+4\Delta \log \log N
\lesssim
\frac{\log N}{\log \nicefrac{1}{1-p_*}}
\lesssim
\frac{\log N}{\log^{-\nicefrac{1}{3}}N}
\]
so in particular $D \leq 2\log^{\nicefrac{4}{3}} N$
which implies  the claimed uniform bound almost surely. 
Hence, any exploration of an edge finds an existing vertex in the tree with probability bounded by
\[
\delta=\frac{2 \Delta}{ \log^{\nicefrac{8}{3}} N}.
\]

Any nice vertex \eqref{eq_nice_vertex} can explore $h$ generations without seeing a tree, and the maximum weight in this generation is almost surely bounded by
\[
M=(p^*)^{h}%\leq(1-p_*)^{h}
\leq
\exp\left(
-\frac{\log^{\nicefrac{2}{3}}N}{9 \log (\Delta I)}
\right).
\]

This means that each generation between $h$ and $D$ finds cycle mass which can be bounded by the weighted Bernoulli sum $\sum_i w_i B_i$ for $B_i \sim {\rm Ber}(\delta)$ i.i.d. and $w\leq M$. By a Bernstein inequality:
\[
\begin{split}
\mathbb{P}\left(\sum_i w_i B_i \geq \frac{1}{\log^{\nicefrac{7}{3}} N}+\delta\right)
&\leq
\exp\left(
-\frac{\frac{1}{2\log^{\nicefrac{14}{3}} N}}{M\delta(1-\delta)+ \frac{M}{3\log^{\nicefrac{7}{3}} N} }
\right)\\
&\leq \exp\left(
-M^{-1+o(1)}
\right).
\end{split}
\]
%\JFtodo{not sure I understand how this works over $\Pab$ -- we certainly can't take a union bound over an uncountable set and the algorithm does change
%
%this does need to apply to every $\Pab$ chain to get high probability understanding of the chain conditional on the graph 
%
%there is always the cop-out option of letting $\cD$ just be the distance to uniform of the marginal of graph $\times$ walk -- i.e. averaging over the graph}
%
%\JFtodo{
%however it's the same cycles of course and the weight bound is uniform so not much changes... can put the explorations into breadth-first order and in a universal cover to couple all explorations from a vertex at once
%
%There is some maximal shape in the universal tree but maybe it includes N-1 ways to thin every label (which is a lot)
%
%If I do want to bound cycle weight over all $p$ it's probably some CM result rather than through rooted trees
%}

So by the union bound over $D-h=O(\log^{\nicefrac{4}{3}} N)$ generations and $\log^{3I}N$ representative vectors $p_g$ on the grid, this is with high probability not seen for all nice vertices and all generations and any $\{p_g:g\in G\}$. 
By adding up the cycle mass we have the claimed bound on total cycle weight at the representative, but then we constructed the trees in Definition \ref{def_T_construction} using the representative so the only change between $p$ and ${ r}(p)$ is in the weights given and that is uniformly $(1+o(1))$.
\end{proof}

We now check that the tree construction's imposed maximal additional depth $4\Delta \log \log N$ is unlikely to be needed at any point.

\begin{proposition}\label{prop_good_alg}
With probability $1-e^{-\sqrt{N}}$, every tree exploration $(T_x)_x$ terminates before the maximal depth and so we achieve the total variation bound
\[
\sup_{p \in \Pab}
\max_{x \in [N]}
\de_{\rm TV}\left(
w\left(\mathcal{E}_x^{(\cdot)}\right),
\pi
\right)
\lesssim
\frac{1}{\log^2 N}.
\]
\end{proposition}

\begin{proof}
%In the previous proposition we found at most $\nicefrac{ N}{\log^{4+4\Delta \log (\Delta I)} N}$ unexplored vertices after stage $2$. %Of course, stage $3$ increases this to at most $\nicefrac{N \Delta I}{\log^2 N}$.

In \ref{stage_3} of the construction, we have maximal weight $ \tilde M=\nicefrac{\log^{4+4\Delta \log (\Delta I)} N}{p_* N}$ and recall via \eqref{eq_separation} that we have separation distance after $1$ timestep $s(1)\leq 1-\nicefrac{1}{\Delta}$. 

In the graph, exploring a set of vertices with total weight $w$ and maximum weight $\tilde{M}$ finds new vertices with, again by a Bernstein inequality,
\[
\begin{split}
\mathbb{P}(\text{less than }\nicefrac{w\pi(i)}{2\Delta} \text{ of type } i)
&\leq
\exp\left(
-\frac{\nicefrac{w^2\pi(i)^2}{8\Delta^2}}{\tilde{M}+\nicefrac{\tilde{M}w\pi(i)}{6\Delta}}
\right)\\
&\leq
\exp\left(
-\frac{\nicefrac{w^2 p_i^2}{8\Delta^2}}{\tilde{M}+\nicefrac{\tilde{M}w p_i}{6}}\right)
\leq
\exp\left(
-\frac{w^2 p_*^3 N}{10\Delta^2 \log^{4+4\Delta \log (\Delta I)} N}
\right).
\end{split}
\]

Note that if \ref{stage_4} is still continuing then we have $w \geq \nicefrac{1}{\log^2 N}$ and so we can claim this with high probability never happens for any generation, type and root vertex with a failure probability which is still exponential (up to a polylogarithmic factor in the exponent). We remove this factor in the claimed stretched exponential.

Outside of this failure event, in \ref{stage_4} we will iteratively couple what weight we can to $\pi_{r(p)}$, allowing a discretisation error, before exploring the other vertices. This terminates before exploring an additional depth at most $k$, where
\[
\left(
1-\frac{1}{2\Delta}
+\tilde{M}I 
\right)^k \leq 
\frac{1}{\log^2 N} 
\]
and so we can take $k=4 \Delta \log \log N $. We can finally say this is close to the real stationary distribution at $p$ through
\[
\de_{\rm TV}\left(
w\left(\mathcal{E}_x^{(\cdot)}\right),
\pi
\right)
\leq
\de_{\rm TV}\left(
w\left(\mathcal{E}_x^{(\cdot)}\right),
\pi_{{ r}(p)}
\right)
+
\de_{\rm TV}\left(
\pi_{{ r}(p)},
\pi
\right)
\]
which requires a trivial calculation:
\[
\de_{\rm TV}\left(
\pi_{{ r}(p)},
\pi
\right)\leq
\sum_{i=1}^I \sum_{\al \in \D} \frac{k(\al)}{N} \al_i^+ \left| p_i -{ r}(p)_i
\right|
\leq
\tfrac{1}{A}\log^{-\nicefrac{8}{3}}N,
\]
referring to Remark \ref{remark_rel_error}.
\end{proof}

\begin{definition}[Nice paths]
A nice path $\mathfrak{p} \in \mathcal{P}^t_{v_0,v_t}$ has:
\begin{itemize}
\item weight $\mathbf{w}(\mathfrak{p})\leq \frac{1}{N \log^2 N}$;
\item first $t-h$ steps in $T_{v_0}$;
\item at the end of this path of $t-h$ steps, the $(t-h)$\textsuperscript{th} vertex is nice \eqref{eq_nice_vertex}.
\end{itemize}
\end{definition}

\begin{proposition}\label{prop_summing_nice_paths}
If we reduce the transition probability to only the sum over nice paths
\[
\left[P_0^t\right]_{x\rightarrow y}
:=
\sum_{\stackrel{\mathfrak{p}\in \mathfrak{P}^t_{x,y}}{\mathfrak{p} \text{ nice}}}
\mathbf{w}(\mathfrak{p})
\]
then we find at $t \sim t_* = \frac{\log N}{\mu}$ and for any $\epsilon>0$
\[
\max_{x,y \in [N]}
N \left[P_0^t\right]_{x\rightarrow y}
\leq 1+\log^{-\nicefrac{1}{4}} N
\]
with probability $1-e^{-\log^{\nicefrac{4}{3}}N}$.
\end{proposition}

\begin{proof}
%Construct the tree $T_x$ as an exploration of the forward ball at $x$, by iteratively adding the \emph{undiscovered} children of the vertex of highest weight. By only adding previously undiscovered children this is forced to be a tree structure, and it continues until the maximal weight is less than $\nicefrac{\log N}{N}$.

%With a generous allowance for deleted cycle edges, a tree with $k$ internal nodes must still have $(I-1)k$ leaves \JFtodo{still can point to the same vertex at the end?}

%%After this procedure, each leaf has weight at least $\nicefrac{p_* \log N}{N}$ and the weights of leaves sum to at most $1$ (less than $1$ iff a cycle was explored), and so this set $\mathcal{E}_x$ which is the set of leaves of $T_x$ has
%\[
%|\mathcal{E}_x|\leq  \frac{N }{p^* \log N}.
%\]

Recall that we write $\mathcal{E}_x$ for the set of unexplored tails (out-edges) at the end of the $T_x$ construction and $\mathcal{E}_x^{(1)}, \dots, \mathcal{E}_x^{(I)}$ for the sets of tails of each type.

Similarly, explore the backwards ball boundary $\mathcal{F}$ of radius $h$ around the target $y$, which by construction has $|\mathcal{F}|\leq N^{\nicefrac{1}{10}}$. Further, reduce this ball to those which have a unique path to $y$ (of length $h$). The boundary leaves of this ball have heads of every type in the sets $\mathcal{F}_y^{(1)}, \dots, \mathcal{F}_y^{(I)}$.

Then, writing $w(.)$ for the weight of the path to either $x$ (for a tail $a$ in $\mathcal{E}$) or $y$ (for a head $b$ in $\mathcal{F}$), we have
\[
\left[P_0^t\right]_{x\rightarrow y}
=
\sum_{i=1}^I \left[P_0^t\right]_{x\rightarrow y}^{(i)},
\quad
\left[P_0^t\right]_{x\rightarrow y}^{(i)}:=
\sum_{a \in \mathcal{E}_x^{(i)}} \sum_{b \in \mathcal{F}_y^{(i)}}
\frac{w(a)w(b)}{p_i}
\mathbbm{1}_{\frac{w(a)w(b)}{p_i}\leq \frac{1}{N \log^2 N}}
\mathbbm{1}_{\omega^{(i)}(a)=b},
\]
where also $\omega^{(i)}$ is the uniform random bijection between heads and tails of layer $i$. Write 
\[
Z^{(i)}
=
Z^{(i)}(\mathcal{E},\mathcal{F})
:=
\mathbb{E}\left(
\left[P_0^t\right]^{(i)}_{x\rightarrow y}
\Big|
\cE^{(i)}, \cF^{(i)}
\right)
\]
 for this function of the random graph conditioned on the two balls and $Z=\sum_{i=1}^I Z^{(i)}$ for the full nice path probability. We show that 
\[
%N\mathbb{E}(Z)=
N Z=
N \mathbb{E}\left(\left[P_0^t\right]_{x\rightarrow y}\Big| \mathcal{E},\mathcal{F}\right)\leq 1 + \frac{2}{\log^2 N}
\] 
almost surely, using Propositions \ref{prop_cycle_weight} and \ref{prop_good_alg}:
\[
\begin{split}
\mathbb{E}\left(\left[P_0^t\right]^{(i)}_{x\rightarrow y}\Big| \mathcal{E},\mathcal{F}\right)
&\leq
\frac{p_i}{\sum_{\al \in \D} k(\al) \al_i^+
 -\frac{2 N}{ \log^{\nicefrac{8}{3}} N}
 -N^{\nicefrac{1}{10}}}
\sum_{a \in \mathcal{E}_x^{(i)}}
\frac{w(a)}{p_i}
\sum_{b \in \mathcal{F}_y^{(i)}}
\frac{w(b)}{p_i}\\
&\leq
\left(\frac{p_i}{\sum_{\al \in \D} k(\al) \al_i^+}
+\frac{\frac{2 N}{ \log^{\nicefrac{8}{3}} N}
 +N^{\nicefrac{1}{10}}}{
\left(
\sum_{\al \in \D} k(\al) \al_i^+
\right)^2 
 }
\right)
\left(
\frac{\pi(i)}{p_i}
+\frac{1}{\log^2 N}
\right)
\sum_{b \in \mathcal{F}_y^{(i)}}
\frac{w(b)}{p_i}\\
&\leq
\left(\frac{p_i}{\sum_{\al \in \D} k(\al) \al_i^+}
+\frac{2 }{ N \log^{\nicefrac{8}{3}} N}
\right)
\left(
\frac{1}{N}
\sum_{\al \in \D} k(\al) \al_i^+
+\frac{1}{\log^2 N}
\right)
\sum_{b \in \mathcal{F}_y^{(i)}}
\frac{w(b)}{p_i}\\
&\leq
\left(
1+\frac{2\Delta}{A \log^{\nicefrac{8}{3}} N}
\right)
\left(
1+\frac{1}{\log^2 N}
\right)
\frac{1}{N}
\sum_{b \in \mathcal{F}_y^{(i)}}
w(b)
\end{split}
\]
and the conclusion follows because $\sum_{i=1}^I \sum_{b \in \mathcal{F}_y^{(i)}}
w(b) \leq 1$. 
Then
\[
\mathbb{P}\left(
Z-\mathbb{E}\left( Z 
\big|
\mathcal{E},\mathcal{F} \right) \geq \frac{\epsilon }{N}
\right)
\leq
\sum_{i=1}^I
\mathbb{P}\left(
Z^{(i)}-\mathbb{E}\left( Z^{(i)} 
\big|
\mathcal{E}^{(i)},\mathcal{F}^{(i)} \right) \geq \frac{\epsilon }{NI}
\right)
\]
so that for each bijection we can apply \cite[Proposition 1.1]{chatterjee07}
\[
\begin{split}
\mathbb{P}\left(
Z^{(i)}-\mathbb{E}\left( Z^{(i)} 
\big|
\mathcal{E}^{(i)},\mathcal{F}^{(i)} \right) \geq \frac{\epsilon }{NI}
\right)
&\leq
\exp\left(
-\frac{\nicefrac{\epsilon^2}{(NI)^2} }{\frac{2}{N \log^2 N} (2 \mathbb{E}\left( Z^{(i)} 
\big|
\mathcal{E}^{(i)},\mathcal{F}^{(i)} \right)+\nicefrac{\epsilon}{NI})}
\right)\\
&\leq
\exp\left(
-\frac{\epsilon^2 \log^2 N}{2I^2 (2+3\epsilon) }
\right)
\end{split}
\]
and the conclusion follows from the union bound over $N(N-1)$ directed pairs, setting $\epsilon=\log^{-\nicefrac{1}{4}} N-2\log^{-2} N$, the second term to offset the error in the mean.
\end{proof}

Of the three conditions defining a nice path, the first is the most important, which we make precise in the following proposition. Recall here that $A$ is the adjacency matrix of the random graph.

\begin{proposition}\label{prop_path_not_nice}
Whenever $t \sim t_*= \frac{\log N}{\mu}$,
\[
\sup_{p \in \Pab}
\max_{v \in V}\left(
\mathbb{P}_v\left(
t \text{ \rm step walk is not nice}
\Big| A
\right)
-
Q_{v,t}\left(
\frac{1}{N\log^2 N}
\right)
\right)
\]
\[
\lesssim
\frac{10^5 \Delta^2  \sqrt{\log \Delta I}}{A^{\nicefrac{3}{2}}}
+\frac{2 \Delta \left( \log \log N \right)^2 \sqrt{\log \Delta I}}{ \log^{\nicefrac{1}{3}} N}
\]
with probability at least $1-3e^{-\sqrt[3]{\log N}}$.
\end{proposition}

\begin{proof}
In Proposition \ref{prop_cycle_weight} we found that (for every $p \in \Pab$ simultaneously) the probability to form a non-nice path from $v$ is uniformly $o(1)$ over the nice vertex roots with probability $1-e^{-\log^2 N}$.

Further we must control that %the final $h$ steps are the only path up to the sink vertex, which is guaranteed if 
the $(t-h)$\textsuperscript{th} vertex is nice. By Lemma \ref{lemma_rerooting}, this fails with probability 
$
(p^*)^{- h}
$
outside of a polynomially unlikely graph event. Using just that $p_*\geq \log^{-\nicefrac{1}{3}}N$, we have
\[
(p^*)^{- h}\leq\left(
1-\log^{-\nicefrac{1}{3}}N
\right)^h
=e^{-\Theta(\log^{\nicefrac{2}{3}}N)}.
\]

The other way that we can escape the tree $T_v$ is by falling to a low weight before time $t-h$. This occurs with exactly probability
\[
1-Q_{v,t-h}\left(
\frac{\log^{3+4\Delta \log (\Delta I)} N}{p_* N}
\right)
\]
which fortunately with $\varphi=\nicefrac{\log^{3+4\Delta \log (\Delta I)} N}{p_* N}$ we have already controlled uniformly in Proposition \ref{prop_network_path_weights}, with probability $1-2e^{-\sqrt[3]{\log N}}$. It remains then to show that
\begin{equation}\label{eq_unlikely_low_weight}
\frac{-\mu (t-h) - \log \varphi}{\sigma_p \sqrt{t-h}}
\rightarrow  \infty
\end{equation}
for which%, because $h=\Theta(\log N)$ and by \eqref{eq_ergodic_variance},
\[
\frac{-\mu (t-h) - \log \varphi}{\sigma_p \sqrt{t-h}}
\sim
\frac{\mu h}{\sigma_p \sqrt{t-h}}
\geq
\frac{\mu^{\nicefrac{3}{2}} h}{\operatorname{Var}_\pi(f) \sqrt{\log N}}
\sim
\frac{\mu^{\nicefrac{3}{2}} \sqrt{\log N}}{\operatorname{Var}_\pi(f) 10 \log \Delta I}.
\]

Note finally that by Lemma \ref{prop_max_square_entropy}
\[
\operatorname{Var}_\pi(f)\leq
\mathbb{E}_\pi\left(
\log^2 p_i
\right)
\leq
\Delta
\sum_i p_i \log^2 p_i
\leq 4\Delta + \Delta \log^2 I
\]
and
\[
\mu^{\nicefrac{3}{2}}\sqrt{\log N} 
\geq 
\left( -p_* \log p_* \right)^{\nicefrac{3}{2}}\sqrt{\log N} 
\gtrsim
\left(
\tfrac{A}{3}  \log \log N
\right)^{\nicefrac{3}{2}}
\]
so we have \eqref{eq_unlikely_low_weight} for all $p \in \Pab$.
\end{proof}

Now we are ready to end the section with the proof of our second main theorem.

\begin{theorem}\label{thm_simultaneous_cutoff}
We define
\[
w_p=\sigma_p \sqrt{\frac{\log N}{\mu_p^3}},
\qquad
t_p=\frac{\log N}{\mu_p},
\]
and take $1 \leq A \leq \theta \log^{\nicefrac{1}{3}}N$ for a sufficiently small constant $\theta$. Then, with probability $1-7e^{-\sqrt[3]{\log N}}$ we find
\[
%\lim_{N \rightarrow \infty}
\sup_{p \in \Pab}
\sup_{\lambda \in \mathbb{R}}
\left|\mathcal{D}(t_p+\lambda w_p)-1+\Phi(\lambda)\right|
\lesssim
\frac{2\cdot 10^5 \Delta^2  \sqrt{\log \Delta I}}{A^{\nicefrac{3}{2}}}
+\frac{4 \Delta \left( \log \log N \right)^2 \sqrt{\log \Delta I}}{ \log^{\nicefrac{1}{3}} N}.
\]
\end{theorem}

\begin{proof}
We prove the result for $|\lambda|^{13}\leq \log N$, and using the monotonicity (for $x>0$)
\[
\begin{split}
\mathcal{D}(x+t_p+ w_p\log^{\nicefrac{1}{13}}N)&<\mathcal{D}(t_p+ w_p\log^{\nicefrac{1}{13}}N),\\
\mathcal{D}(-x+t_p- w_p\log^{\nicefrac{1}{13}}N)&>\mathcal{D}(t_p- w_p\log^{\nicefrac{1}{13}}N),
\end{split}
\]
 it is automatically true outside of that interval with an $e^{-\Omega\left(\log^{\nicefrac{2}{13}}N\right)}$ error.

Recall from Proposition \ref{prop_network_path_weights} that we have (with probability at least $1-2e^{-\sqrt[3]{\log N}}$)
\[
\left|
Q_{i,t}\left(
\varphi
\right)
-\Phi\left(
-\lambda_{\varphi}
\right)
\right|
\lesssim
\frac{10^5 \Delta^2  \sqrt{\log \Delta I}}{A^{\nicefrac{3}{2}}}
+\frac{2 \Delta \left( \log \log N \right)^2 \sqrt{\log \Delta I}}{ \log^{\nicefrac{1}{3}} N}
\]
for the Gaussian quantile
\[
\lambda_{\varphi}=
\frac{\mu t + \log \varphi}{\sigma_p \sqrt{t}}
\]
which we will use to determine the distance from stationarity. We have parameters
\begin{equation}\label{eq_phi_lambda}
t=t_p+\lambda w_p,
\quad
\left|
\log \varphi + \log N 
\right|\leq
2\log \log N.
\end{equation}

By Lemma \ref{prop_min_general_variance} for $p \in \Pab$ we have $\sigma_p \geq \sqrt{A} \log^{-\nicefrac{1}{6}}N$, and so
\[
\begin{split}
\left|
\lambda-\lambda_\varphi
\right|
&=
\frac{\left|-\lambda \mu w_p - \log N - \log \varphi+\lambda\sigma_p \sqrt{t_p+\lambda w_p}\right|}
{\sigma_p \sqrt{t_p+\lambda w_p}}\\
&\leq
\frac{2 \log \log N}
{\sigma_p \sqrt{t_p+\lambda w_p}}
+
\lambda
\frac{\left|- \sigma_p \sqrt{t_p}+\sigma_p \sqrt{t_p+\lambda w_p}\right|}
{\sigma_p \sqrt{t_p+\lambda w_p}}\\
&\lesssim
\lambda
\frac{\left|- \sigma_p \sqrt{t_p}+\sigma_p \sqrt{t_p+\lambda w_p}\right|}
{\sigma_p \sqrt{t_p+\lambda w_p}}\\
&\lesssim
\frac{\lambda^2}{2}\frac{w_p}{t_p}=O\left(\frac{\lambda^2}{\sqrt{\log N}}\right)=O\left(\frac{1}{\log^{\tfrac{1}{3}+\tfrac{1}{78}} N}\right)\\
\end{split}
\]
where the final order follows from Lemma \ref{prop_max_variance} and Propositions \ref{prop_positive_p} and then recalling $|\lambda|^{13}\leq \log N$.
%on the conditions:
%\[
%(|\lambda|+\lambda^2) 
%w_* \ll t_*; 
%\qquad
%\log \log N \ll \sigma_L \sqrt{t_*}.
%\]
%
%Note we have $\sigma_L \leq\sqrt{\operatorname{Var}_\pi( \log p_i)}  \lesssim \frac{1}{3}\log \log N$ and hence the first condition is given by
%\[
%\lambda^2\frac{\log^2 \log N}{\log N}\ll \mu \gtrsim \frac{A}{3} \log^{\nicefrac{-1}{3}}N \log \log N
%\]
%by absorbing all but one outcome and so seeing the binary entropy lower bound $\mu \geq H_2(p_*) \geq - p_* \log (p_*)$, and so this is automatic when $\lambda^4 \leq \log N$.

%For the second (using Proposition \ref{prop_max_entropy})
%\[
%\frac{\log N}{\log^2 \log N} \operatorname{Var}_\pi( \log p_i)\gg \mu \leq \log \Delta I
%\]
%we rely on the stated assumption.

The above calculation was insensitive to $\varphi$ in the given window \eqref{eq_phi_lambda} so we can set
\[\varphi_{\rm upper}=\frac{\log^2 N}{N}, \qquad \varphi_{\rm lower}=\frac{1}{N\log^2 N}.\]

We now check that $\mathcal{D}_{v}(t) \approx Q_{v,t}(\varphi)$ in both directions. 
For the lower bound by \cite[Equation 13]{bordenave18}
\[
Q_{v,t}(\varphi_{\rm upper})\leq \mathcal{D}_{v}(t)
+
\sqrt{\frac{1}{\varphi_{\rm upper}}\sum_{ v \in [N]}\frac{1}{N^2}}
=
\mathcal{D}_{v}(t)
+\frac{1}{\sqrt{ N \varphi_{\rm upper}}}.
\]

For the upper bound, using that there is at most one nice path from source to sink, as well as Proposition \ref{prop_summing_nice_paths}
\[
\begin{split}
\max_{v \in V}
\mathcal{D}_{v}(t)&=
\max_{v \in V}
\sum_w \left(
\frac{1}{N}-\left[P^t\right]_{v\rightarrow w}
\right)^+\\
&\leq
\max_{v \in V}
\sum_w \left(
\frac{1+\log^{-\nicefrac{1}{4}} N}{N}-\left[P_0^t\right]_{v\rightarrow w}
\right)^+\\
&\leq
\log^{-\nicefrac{1}{4}} N+
\max_{v \in V}
\mathbb{P}_v\left(
t\text{ step path not nice}
\right)
\end{split}
\]
with probability $1-e^{-\log^{\nicefrac{4}{3}}N}$.

%given that $\max_{x,y \in [N]}
%\left[P_0^t\right]_{x\rightarrow y}\leq\frac{1+\epsilon}{N}
%$ from . Precisely, this gives  that
%\[
%\max_{v \in V}
%\mathcal{D}_{v}(t)
%\leq
%\max_{v \in V}
%\mathbb{P}_v\left(
%t\text{ step path not nice}
%\right)
%+
%\log^{-\nicefrac{1}{4}} N
%\]

We conclude for $s = \left\lfloor\sqrt{\log N}\right\rfloor$ using Proposition \ref{prop_path_not_nice} that with probability $1-4e^{-\sqrt[3]{\log N}}$
\[
\begin{split}
\max_{v \in [N]}
\mathcal{D}_{v}(t)
&\leq
\max_{v \in V}
\mathcal{D}_{v}(t-s)
+(p^*)^{s}\\
&\lesssim
Q_{v,t}\left(
\varphi_{\rm lower}
\right)
+\frac{10^5 \Delta^2  \sqrt{\log \Delta I}}{A^{\nicefrac{3}{2}}}
+\frac{2 \Delta \left( \log \log N \right)^2 \sqrt{\log \Delta I}}{ \log^{\nicefrac{1}{3}} N}
+(p^*)^{s}
\end{split}
\]
and the final term has smaller order as $p^*\leq 1-\log^{-\nicefrac{1}{3}}N$.
\end{proof}

%\clearpage
\section{Optimisation}\label{sec_optimisation}

In this section we want to compare the full space of possible probability vectors $\sP$, which contains of course some arbitrarily small probabilities violating our mixing requirements. $N$ is still taken sufficiently large in all statements.

We note first that at least the optimiser cannot be close to the boundary.

\begin{proposition}\label{prop_positive_p}
If there is some strictly positive vector in $\mathscr{P}$, then we find
\[
\min_{i \in [I]}
\argmax_{p \in \mathscr{P}}
\mu_p\geq\frac{1}{(\Delta I)^{\Delta^2}}
\]
\end{proposition}

\begin{proof}
Define
\[
{f}(p)=\begin{cases}
p \log \frac{1}{p} & p>0\\
0 & p \leq 0
\end{cases}
\]
We find the stationary points of Lagrangian
\[
\mathcal{L}=\sum_{i=1}^I\sum_{\al \in \D} \frac{k(\al)}{N} \al_i^+ {f}(p_i)
+\sum_{\al \in \D}
\lambda_\alpha^+ 
\left(\sum_{i=1}^I
1-p_i\alpha_i^+
\right)
+\sum_{\al \in \D}
\lambda_\alpha^-
\left(\sum_{i=1}^I
1-p_i\alpha_i^-
\right)
\]
leading for each $i$ to either $p_i\leq 0$ or
\[
p_i=\exp\left(
-\sum_{\al \in \D}
\lambda_\alpha^+ 
\frac{\alpha_i^+}{e m_i}
-\sum_{\al \in \D}
\lambda_\alpha^-
\frac{\alpha_i^-}{e m_i}
\right)
\]
where $m_i$ is the mean degree of layer $i$.

$\mu$ is concave on $\mathscr{P} \ni p$ where $p_i>0$, and further
\[
(\nabla \mu)_i\Big|_{p_i=0^+}=\infty
\]
so the directional derivative from a point on an axis in the positive direction of must be positive: the optimiser $p$ must fall in the strictly positive orthant.

Then observe that with $eC:=\sum_{\al \in \D}\lambda_\alpha^++\sum_{\al \in \D}\lambda_\alpha^-$, for every $i\in [I]$ we have
\[
e^{-C\Delta} \leq p_i \leq e^{-\nicefrac{C}{\Delta}}
\]
but also for some $\al \in \D$
\[
1=\sum_{i=1}^I p_i \alpha^+_i\leq \Delta I e^{-\nicefrac{C}{\Delta}}
\implies
e^{-C\Delta}\geq\frac{1}{(\Delta I)^{\Delta^2}}.
\]
%and so the above zero derivative claim is
%\[
%\sum_{i=1}^I\sum_{\al \in \D} \frac{k(\al)}{N} v_i \al_i^+=\sum_{i=1}^I\sum_{\al \in \D} \frac{k(\al)}{N} v_i \al_i^+ \log \frac{1}{p_i}
%\]
\end{proof}

The space $\sP$ also might contain vectors close to constant which would break the Gaussian cutoff picture, but again (because we exclude the regular graph) this will not happen at the optimiser.

\begin{corollary}\label{cor_min_variance}
When $p$ is the optimiser of $\mu$ in $\mathscr{P}$ which contains a strictly positive vector, and the degrees defining $\mathscr{P}$ satisfy Assumption \ref{ass_different_outdegree}, we find
\[
\mathbb{V}{\rm ar}_\pi \log p_i \geq \frac{1}{(\Delta I)^{\Delta^2}}.
\]
\end{corollary}

\begin{proof}
By the triangle inequality for any $X_1,X_2 \stackrel{\rm i.i.d.}{\sim}X$
\[
\mathbb{E}\left((X_1-X_2)^2\right)\leq 2 \mathbb{V}{\rm ar}(X),
\]
and therefore for any $i,j \in [I]$ we can let $X$ have the distribution $\pi$ to see
\[
\begin{split}
\mathbb{V}{\rm ar}_\pi \log p_i
&\geq 
\frac{1}{2}
\mathbb{E}\left((\log p_{X_1}-\log p_{X_2})^2;\left\{\{X_1,X_2\}=\{i,j\}\right\}\right)\\
&=
\pi(i)\pi(j)\left(\log p_i - \log p_j\right)^2\\
&\geq p_i p_j \log^2 \frac{p_i}{p_j}.
\end{split}
\]

Then, by Assumption \ref{ass_different_outdegree} there exist $\al, \beta \in \D$ and $\cdot \in \{+,-\}$ with $\sum \al^\cdot - \sum \beta^\cdot\geq 1$, and for these vectors:
\[
\al \cdot p =1 \implies \min p \leq \frac{1}{\sum \al^\cdot};
\]
\[
\beta \cdot p =1 \implies \max p \geq \frac{1}{\sum \beta^\cdot}.
\]

Setting $j=\argmax p$ then, and using that this lower bound is increasing in $p_j>p_i$,
\[
\mathbb{V}{\rm ar}_\pi \log p_i 
\geq
\frac{p_i}{\sum \beta^\cdot} \log^2 \left( p_i \sum \beta^\cdot \right).
\]

We use Proposition \ref{prop_positive_p} and that over $p_i<\sum \beta^\cdot$ this lower bound has a unique local maximum
\[
\begin{split}
\mathbb{V}{\rm ar}_\pi \log p_i 
&\geq
\frac{1}{\left(\sum \al^\cdot \right) \left(\sum \beta^\cdot \right)} \log^2 \left( \frac{\sum \beta^\cdot}{\sum \al^\cdot}  \right)
\wedge
\frac{1}{(\Delta I)^{\Delta^2}\sum \beta^\cdot}\log^2 \left( \frac{\sum \beta^\cdot}{(\Delta I)^{\Delta^2}}  \right)\\
&\geq
\frac{1}{\Delta^2} \log^2 \left( \frac{\Delta-1}{\Delta}  \right)
\wedge
\frac{1}{\Delta(\Delta I)^{\Delta^2}}\log^2 \left( \frac{\Delta}{(\Delta I)^{\Delta^2}}  \right)\\
\end{split}
\]
and then it just remains to remove the larger case. Applying standard logarithmic inequalities:
\[
\frac{1}{\Delta^2} \log^2 \left( \frac{\Delta-1}{\Delta} \right)
\geq 
\frac{1}{\Delta^4} ;
\]
\[
\frac{1}{\Delta(\Delta I)^{\Delta^2}}\log^2 \left( \frac{\Delta}{(\Delta I)^{\Delta^2}}  \right)
\leq
\frac{\Delta^3}{(\Delta I)^{\Delta^2}}\log^2 \left( \Delta I  \right)
\leq
\frac{\Delta^3}{(\Delta I)^{\Delta^2-2}}.
\]

Immediately from $I\geq 1$ we see that the second bound is smaller when $\Delta \geq 3$, and when $\Delta=2$ we use that $8I^4 \log^2 2 \geq \log^2 (8I^4)$ for the same observation.

%\JFtodo{this could be easier if we're comparing to the simplified expression}

Finally we simplify the expression
\[
\mathbb{V}{\rm ar}_\pi \log p_i
\geq \frac{\Delta^2-1}{\Delta(\Delta I)^{\Delta^2}}\log^2 \left( \Delta I \right)
\geq \frac{\Delta-1}{(\Delta I)^{\Delta^2}}\log^2 \left( \Delta I \right)
\geq \frac{1}{(\Delta I)^{\Delta^2}}.
\]
\end{proof}

The final comment before the proof of our first main theorem is to exclude probabilities close to $1$ -- finding these chains in $\sP$ is only possible with a layer consisting only of cycles and in that they case have very slow mixing.

\begin{proposition}\label{prop_max_prob}
Any vector $q \in \sP$ with 
\[
q^*\geq 1-\frac{1}{3(\Delta I)^{\Delta^2}\log(\Delta I) }
\]
has at almost surely at time $t \sim t_p$
\[
\inf_{v \in [N]}
\mathcal{D}_v%^{(q)}
(t)
\geq 1-\frac{1}{\sqrt{N}}.
\]
\end{proposition}

\begin{proof}
Because $\sum_i q_i\leq 1$ and $q^* > \nicefrac{1}{2}$, there is a unique $i$ with $q_i=q^*$. Also, for every type $\al \in \D$ we have $\al_i^+=\al_i^-=1$ at this $i$.

By taking the $i$ step we see no increase in the set of possible positions of the walker, and so the size of this set at time $t$ is stochastically dominated by
\[
\left(\Delta I\right)^{B_t},
\qquad
\text{where}\,\,
B_t \sim \operatorname{Bin}\left(
t,1-q_i
\right).
\]

So at any time 
\[
t \sim t_p \leq (\Delta I)^{\Delta^2}\log N
\qquad
\text{we have} \,\,
B_t \leq \frac{1+o_{\p}(1)}{3\log(\Delta I)} \cdot \log N
\]
and so the walker is distributed on $N^{\nicefrac{1}{3}+o_{\p}(1)}$ vertices. This bound holds almost surely regardless of the graph realisation or initial vertex (the order in probability is just considering the randomness in $B_t$) which gives the claimed almost sure statement.
\end{proof}

We now prove our first main theorem, which was as follows.

\optimisercutoff*

\begin{proof}
We first comment that by Proposition \ref{prop_max_prob} we may restrict our attention to $q \in \sP$ with $q^*\leq 1-\nicefrac{1}{3(\Delta I)^{\Delta^2}\log(\Delta I) }$.

Write $p=p_\bigstar$. We compare the mixing of the walk defined by $p$ to that of every other $q \in \mathscr{P}$, on the same single graph realisation, through various cases. 
 These cases classify $q$ by its vector of relative errors \[\epsilon_i:=|1-q_i/p_i|.\]
 
 For some sufficiently small constant $\theta$, we control that the supremum error tends to $0$ with small $\theta$ in all four cases of the relative error, with high probability. Then the conclusion will follow from taking $\theta$ arbitrarily small.
 
Notationally a particular time $t$ is parametrised
\[
t=t_q+\lambda_q w_q=t_p+\lambda_p w_p
\]
so here really $\lambda_q=\lambda(t,q)$.

Using Theorem \ref{thm_simultaneous_cutoff}, then, we have control of the $q$ curve underneath the $p$ curve if $\lambda(t,q)>\lambda(t,p)$. To this end, we show the following.

\begin{claim}
%Let $\theta$ be a sufficiently small constant. 
Then, uniformly over all $N$ sufficiently large and all $q \in \mathscr{P}$ with $\|\epsilon\|_{2,\pi}^2 \geq \nicefrac{\theta}{\sqrt{\log N}}$ and $\|\epsilon\|_{\infty}\leq \theta^2$ we have
\[
\sup_t \left(
\Phi(\lambda_q)
-
\Phi(\lambda_p)
\right)^+
= e^{-\Omega\left( \theta^{-2}\right)}.
\] 
\end{claim}

\begin{proof}[Proof of claim]
Note first that by construction $t_q>t_p$. Hence for $q$ to be relevant in $\inf_{\sP}
\mathcal{D}(t)$ it will require at least $w_q>w_p$.
Suppose we are in this case, and at some time $t$ the Gaussians have the same cumulative probability, i.e.
\[
t=
t_p+\lambda w_p
=
t_q+\lambda w_q
\]

Then we claim that this intersection has big negative $\lambda$, and so it remains to lower bound
\[
-\lambda=
\frac{t_q-t_p}{w_q-w_p}
\]
which would say that at this point both curves are anyway close to $1$.

By Lemma \ref{prop_quad_mu} for the numerator
\[
\frac{t_q-t_p}{\log N}
=\frac{1}{\mu_q}-\frac{1}{\mu_p}
\geq 
\frac{1}{\mu_p - \frac{1}{3} \|\epsilon\|_{2,\pi}^2}-\frac{1}{\mu_p}
\geq
\frac{\|\epsilon\|_{2,\pi}^2}{3\mu_p^2}.
\]

We then note from Proposition \ref{prop_sigma_approx} that as $\|\epsilon\|_{\infty}\rightarrow 0$
\[
\sigma_q-\sigma_p=
\frac{
\sigma^2_q-\sigma^2_p}{
\sigma_q+\sigma_p}
\lesssim
\frac{11 I \Delta^3 \log^2 (I \vee 8) \|\epsilon\|_{\infty}}{2 \sigma_p} 
\]
hence for the denominator
\[
\begin{split}
\frac{w_q-w_p}{\sqrt{\log N}}&=
\frac{\sigma_q}{\mu_q^{\nicefrac{3}{2}}}
-
\frac{\sigma_p}{\mu_p^{\nicefrac{3}{2}}}
\lesssim
\frac{\sigma_p+\frac{11 I \Delta^3 \log^2 (I \vee 8) \|\epsilon\|_{\infty}}{2 \sigma_p} }{
\left(
\mu_p - \|\epsilon\|_{2,\pi}^2
\right)^{\nicefrac{3}{2}}
}
-
\frac{\sigma_p}{\mu_p^{\nicefrac{3}{2}}}\\
&=
\frac{\sigma_p}{\mu_p^{\nicefrac{3}{2}}}\left(
\left(
1+\frac{11 I \Delta^3 \log^2 (I \vee 8) \|\epsilon\|_{\infty}}{2 \sigma^2_p}
\right)
\left(
1-\frac{\|\epsilon\|_{2,\pi}^2}{\mu_p}
\right)^{\nicefrac{-3}{2}}
-
1
\right)\\
&\sim
\frac{\sigma_p}{\mu_p^{\nicefrac{3}{2}}}\left(
\frac{11 I \Delta^3 \log^2 (I \vee 8) \|\epsilon\|_{\infty}}{2 \sigma^2_p}
+
\frac{3\|\epsilon\|_{2,\pi}^2}{2\mu_p}
\right).
\end{split}
\]
%where we required $\epsilon \ll \nicefrac{\sigma^2_{L}(p)}{\log^2 \log N}$.

And so, bringing both inequalities together and uniformly in $q$ (recall Corollary \ref{cor_min_variance} and Proposition \ref{prop_positive_p}),
\[
\begin{split}
-\lambda&\geq 
\frac{\|\epsilon\|_{2,\pi}^2 \sqrt{\log N} }{3\sigma_p\sqrt{\mu_p}\left(
\frac{11 I \Delta^3 \log^2 (I \vee 8) \|\epsilon\|_{\infty}}{2 \sigma^2_p}
+
\frac{3\|\epsilon\|_{2,\pi}^2}{2\mu_p}
\right)}
\geq \Omega(1)
\frac{\|\epsilon\|_{2,\pi}^2 \sqrt{\log N} }{
\frac{ \|\epsilon\|_{\infty}}{\sigma_p}
+
\frac{\|\epsilon\|_{2,\pi}^2}{\sqrt{\mu_p}}}\\
&\geq \Omega(1)\left(
\sqrt{\mu_p\log N}
\wedge
\frac{\|\epsilon\|_{2,\pi}^2 \sigma_p \sqrt{\log N} }{
 \|\epsilon\|_{\infty}}
\right)
= \Omega\left(
\frac{1}{\theta}
\right).
\end{split}
\]

We conclude
\[
\sup_t \left(
\Phi(\lambda_q)
-
\Phi(\lambda_p)
\right)^+
=
\sup_{\lambda_q<\lambda} \left(
\Phi(\lambda_q)
-
\Phi(\lambda_p)
\right)^+
\leq 
\Phi(\lambda)
\leq
\frac{1}{-\lambda\sqrt{2\pi}}e^{-\nicefrac{\lambda^2}{2}}
\] 
which has the claimed asymptotic as $\theta \rightarrow 0$.
%\[
%\begin{split}
%&\geq
%\frac{1-2\epsilon}{\epsilon\sqrt{\mu_p}}
%\left(
%\frac{\Delta  \log \log N}{18} 
%\sqrt{\frac{\Delta \log N}{B}}
%+
%\frac{\sigma_L(p)}{\mu_p}
%(3+3\log I)
%\right)^{-1}\\
%&\geq
%\frac{1-2\epsilon}{\epsilon\sqrt{\log I}}
%\left(
%\frac{\Delta  \log \log N}{18} 
%\sqrt{\frac{\Delta \log N}{B}}
%+
%\frac{\tfrac{1}{3}\log \log N}{\tfrac{A}{3}\log^{\nicefrac{-1}{3}}N \log \log N}
%(3+3\log I)
%\right)^{-1}\\
%&\geq \log^2 N \text{ given } \epsilon \leq \tfrac{1}{A}\log^{-\nicefrac{8}{3}}N
%\end{split}
%\]
\end{proof}

This required a lower bound assumption on the relative error, but we find that without that assumption there is not a significant difference between $p$ and $q$.

\begin{claim}
%Let $\theta$ be a sufficiently small constant. 
Then, uniformly over all $N$ sufficiently large and all $q \in \mathscr{P}$ with   $\|\epsilon\|_{2,\pi}^2 \leq \nicefrac{\theta}{\sqrt{\log N}}$ and $\|\epsilon\|_{\infty}\leq \theta^2$, 
we have
%at the same time, i.e. when
%\[
%t_p+\lambda_p w_p
%=
%t_q+\lambda_q w_q
%\]
\[
\sup_t \left(
\Phi(\lambda_q)
-
\Phi(\lambda_p)
\right)^+
=O\left(
\theta
\right).
\] 
\end{claim}

\begin{proof}[Proof of claim]

By Lemma \ref{prop_max_entropy}, Corollary \ref{cor_min_variance} and Lemma \ref{lemma_dynamical_variance_bounds}
\[
w_p=\sigma_p \sqrt{\frac{\log N}{\mu_p^3}}
\geq\frac{1}{\Delta(\Delta I)^{\Delta^2}} \sqrt{\frac{\log N}{\log(\Delta I)^3}}
\]
and by Lemma \ref{prop_quad_mu} and Proposition \ref{prop_positive_p}
\[
\left|
t_p-t_q
\right|
\leq
\frac{\log N}{\mu_p - \|\epsilon\|_{2,\pi}^2}-\frac{\log N}{\mu_p}
\lesssim
\frac{\|\epsilon\|_{2,\pi}^2\log N}{\mu_p^2}
\leq
(\Delta I)^{2\Delta^2}\|\epsilon\|_{2,\pi}^2\log N
\]
and so by our assumption on $\|\epsilon\|_{2,\pi}^2$
\[
\frac{\left|
t_p-t_q
\right|}{w_p}
\leq
\|\epsilon\|_{2,\pi}^2
\sqrt{\log N}
\cdot
\Delta
(\Delta I)^{3\Delta^2} \log^{\nicefrac{3}{2}}(\Delta I)
=O\left(
\theta
\right).
\]

Comparing the window lengths, in the previous proof we found
\[
\left|
w_q-w_p
\right|
\lesssim
\frac{\sigma_p}{\mu_p^{\nicefrac{3}{2}}}\left(
\frac{11 I \Delta^3 \log^2 (I \vee 8) \|\epsilon\|_{\infty}}{2 \sigma^2_p}
+
\frac{3\|\epsilon\|_{2,\pi}^2}{2\mu_p}
\right)\sqrt{\log N}.
\]
and by inserting the lower bound on $w_p$ and applying the same constant upper and lower bounds on $\sigma_p$ and $\mu_p$
\[
\left|
\frac{w_q}{w_p}-1
\right|
\leq
\frac{\sigma_p}{\mu_p^{\nicefrac{3}{2}}}\left(
\frac{11 I \Delta^3 \log^2 (I \vee 8) \|\epsilon\|_{\infty}}{2 \sigma^2_p}
+
\frac{3\|\epsilon\|_{2,\pi}^2}{2\mu_p}
\right)
\Delta
(\Delta I)^{\Delta^2}
\log(\Delta I)^{\nicefrac{3}{2}}
=O\left(
\theta^2
\right).
\]

By these calculations
\[
\lambda_p=\frac{t_q-t_p+\lambda_q w_q}{w_p}
=
O\left(
\theta
\right)
+\lambda_q \left(1+
O\left(
\theta^2
\right)
\right)
\]
and so we have control of the difference in the region $|\lambda_q|\leq \nicefrac{1}{\theta}$, in which $|\lambda_p-\lambda_q|=O\left(
\theta
\right)$. When $\lambda_q\leq -\nicefrac{1}{\theta}$, however, $\left(
\Phi(\lambda_q)
-
\Phi(\lambda_p)
\right)^+\leq \Phi(-\nicefrac{1}{\theta})=o\left(
{\theta}
\right)$. When $\lambda_q\geq \nicefrac{1}{\theta}$, we use $\lambda_p \geq (1+o(1)) \lambda_q$ to say $\left(
\Phi(\lambda_q)
-
\Phi(\lambda_p)
\right)^+\leq 1-\Phi(\nicefrac{(1+o(1))}{\theta})$ which is also $o\left(
{\theta}
\right)$.
\end{proof}

These two cases combine, using Theorem \ref{thm_simultaneous_cutoff}, to prove 
\[
%\limsup_{N \rightarrow \infty}
\p\left(
\sup_{t >0}
\left|
\inf_{\|\epsilon\|_{\infty}\leq \theta^2}
\mathcal{D}(t)-1+\Phi\left(
\frac{t-t_\bigstar}{w_\bigstar}
\right)
\right|
\geq \Omega(\theta)
\right)
=O\left(
%\frac{\left( \log \log N \right)^2 }{ \log^{\nicefrac{1}{3}} N}
%+\frac{1}{\theta^{\nicefrac{3}{2}}\sqrt{\log N}}+
e^{-\sqrt[3]{\log N}}
\right)
\]
or in particular the claimed convergence to $0$ in probability, but only for the partial infimum over the smaller space
\[
\left\{q:
\|\epsilon\|_{\infty}\leq \theta^2
\right\}
\subset \mathscr{P}_{\theta \log^{1/3}N}
=
\left\{q:
q_*\geq \theta
\right\}
\]
provided $\theta \leq (\Delta I)^{\Delta^2}$ (recall Proposition \ref{prop_positive_p}).

Outside of these two cases, we are no longer in the central limit regime of  Theorem \ref{thm_simultaneous_cutoff} and so we instead use the following bound.

\begin{theorem}[ {\cite[Theorem A]{kloeckner2019}} ]\label{thm_concentration}
Consider a chain $(X_k)_k$ with transition matrix $P$, stationary distribution $\pi$ and arbitrary initial condition. Suppose also the chain is a contraction (in a norm with Assumptions \ref{assumptions_kloeckner}) with parameter $\delta$
\[
\pi(x)=0 \quad \implies \quad
\frac{\|Px\|}{\|x\|} \leq 1-\delta.
\]

The empirical sum $S_t=\sum_{k=1}^t { f}(X_k)$, with mean term $\mu=\pi({ f})$, 
has
\[
\p\left(\left|
\frac{1}{t} S_t-\mu 
\right|\geq a
\right)
\leq
\exp\left(
1-t \frac{\delta}{14\delta+9}
\frac{a^2}{\|{ f}\|^2}
\mathbbm{1}_{3a \leq \delta \|{ f}\|}
\right)
\]
whenever $t>\left\lceil\frac{\log 100}{\log 13 - \log(1-\delta)}\right\rceil$.
\end{theorem}

\begin{claim}
Every $q \in \mathscr{P}$ with $\|\epsilon\|_{\infty}\geq \theta^2$ but still \[
q_* \geq \frac{\theta \log \log N}{\log N}
\] has $\mathcal{D}(t) = 1- O\left(
\nicefrac{1}{\log \log N}
\right)$ at $t \sim t_p$, with high probability.
\end{claim}

\begin{proof}[Proof of claim]
By Proposition \ref{prop_Qbar_and_q} with probability at least $1-e^{-\sqrt[3]{\log N}}$ (simultaneously for every $p \in \mathscr{P}$)
\[
\min_{v \in V}
q^{(\tau(v))}_{t-2\ell}\left(\varphi\right)-
\frac{2}{\sqrt{\log N}}
\leq
\min_{v \in V}
\bar{Q}_{v,t-2\ell}\left(\frac{\varphi}{\left(1+\frac{1}{\log N}\right)^2} \right)
\]
and then as in the proof of Proposition \ref{prop_network_path_weights}, with probability at least $1-N^{\nicefrac{-1}{6}}$ (coming from Lemma \ref{lemma_rerooting}),
\[
\min_{v \in V}
\bar{Q}_{v,t-2\ell}
\left(\frac{\varphi}{\left(1+\frac{1}{\log N}\right)^2} \right)
\leq
\min_{v \in [N]}
Q_{v,t}
\left(\frac{\varphi q_*^{2\ell}}{\left(1+\frac{1}{\log N}\right)^2} 
\right)
+
(q^* )^{\ell}
\]
where the ball radius to find a nice vertex is
\[
\ell=
\left\lceil
\frac{3 \log \log N}{\log \nicefrac{1}{q^*}}
\right\rceil
\implies
(q^* )^{\ell} \leq \frac{1}{\log^{3} N}
\]
and finally as in the proof of Theorem \ref{thm_simultaneous_cutoff}, almost surely
\[
Q_{v,t}
\left(\frac{\varphi q_*^{2\ell}}{\left(1+\frac{1}{\log N}\right)^2} 
\right)
\leq 
\mathcal{D}_{v}(t)
+\frac{1+\frac{1}{\log N}}{
q_*^{\ell}\sqrt{
\varphi 
 N}}.
\]

So to use this all we have to do is lower bound $\min_{\al \in \D}q^{(\al)}_t$. From the proof of Lemma \ref{prop_quad_mu} and Proposition \ref{prop_positive_p}
\[
\mu_q
\leq \mu_p - \frac{\|\epsilon\|_{2,\pi}^2}{1+\|\epsilon\|_\infty} 
\leq \mu_p - \frac{\|\epsilon\|_\infty^2}{(\Delta I)^{\Delta^2}\left(1+\|\epsilon\|_\infty\right)} 
%=: \mu_p - 2a.
\]

To control a constant error in the sample mean, we use the concentration inequality of Theorem \ref{thm_concentration} with parameters $\delta=\nicefrac{1}{4\Delta}$, $t \sim t_p$, $f(i)=-\log q_i$. By Lemma \ref{prop_max_square_entropy} and the simple observation that $I\geq 2$ forces $q_*\leq \nicefrac{1}{2}$
\[
\log^2 2
\leq
\log^2\frac{1}{q_*}
\leq
\|{ f}\|^2
\]
and so we can set for the small constant $a$
\[
a:=
\left(
\frac{1}{2}\frac{\|\epsilon\|_\infty^2}{(\Delta I)^{\Delta^2}\left(1+\|\epsilon\|_\infty\right)}
\right)
\wedge
\left(
\frac{1}{4\Delta}\frac{\log 2}{3}
\right)
\]
to guarantee the requirement $3a \leq \delta \|f\|$ of Theorem \ref{thm_concentration}. This theorem then controls
\[
1-q^{(\alpha)}_{t-2\ell}(\varphi)=
\mathbb{P}\left(
S_{t-2\ell}
\geq - \log \varphi
\right)
\]
if $- \log \varphi \geq (\mu_q+a)(t-2\ell)$. By construction $\mu_q\leq \mu_p-2a$ and $t-2\ell \sim t_p$ (the latter using our limit on $q^*$ by Proposition \ref{prop_max_prob}), so take
\[
- \log \varphi = (\mu_p-a)t_p
\implies
\varphi
=\frac{1}{N} N^{\nicefrac{a}{\mu_p}}.
\]

Note also
\[
\|{ f}\|^2
 \leq
\left(
\log\frac{1}{q_*}
+
\sqrt{\frac{\Delta\log^2 (I \vee 8)}{q_*}}
\right)^2
 \leq
\frac{2\Delta\log^2 (I \vee 8)}{q_*}
\]
and so at the prescribed value of $\varphi$ we see
\[
\begin{split}
q^{(\alpha)}_{t-2\ell}(\varphi)
&\geq
1-\exp\left(
1-(t-2\ell) q_* \frac{\delta}{14\delta+9}
\frac{a^2}{2\Delta\log^2 (I \vee 8)}
\right)\\
&\geq
1-\frac{2\Delta(14\delta+9)\log^2 (I \vee 8)}{\delta a^2 (t-2\ell) q_*}
=: 1-\frac{C}{(t-2\ell) q_*}.
\end{split}
\]

%\JFtodo{shuffle these last two cases so that this one is 
%\[
%q_* \geq \frac{\theta \log \log N}{\log N}
%\]
%Then we still get some crazy $\theta$ dependence but ultimately it's $\mathcal{D}_{v}(t)\geq 1-\frac{1}{\log \log N}$}

Tracking the above argument, this implies
\[
\begin{split}
\mathcal{D}_{v}(t)
&\geq
1-\frac{C}{ (t-2\ell) q_*}
-\frac{1+\frac{1}{\log N}}{
q_*^{\ell}\sqrt{
\varphi 
 N}}
 -(q^* )^{\ell}
 -
\frac{2}{\sqrt{\log N}}\\
&\gtrsim
1-\frac{C}{ t_p q_*}
-\frac{1}{
q_*^{\ell}N^{\nicefrac{a}{2\mu_p}} }
 -\frac{1}{\log^{3} N}
 -
\frac{2}{\sqrt{\log N}}\\
&\gtrsim
1- O\left(
\frac{1}{\log \log N}
\right)
-\frac{\left(\theta\log N\right)^{\ell}}{
N^{\nicefrac{a}{2\mu_p}} }
 -
\frac{2}{\sqrt{\log N}}.\\
\end{split}
\]
\end{proof}

\begin{claim}
Every $q \in \mathscr{P}$ with $\|\epsilon\|_{\infty}\geq \theta^2$ and some set $[J] \ni j$ having $q_j\leq \nicefrac{\theta \log \log N}{\log N}$ has $\mathcal{D}(t) = 1-O(\theta)$ at $t \sim t_p$, with high probability.
\end{claim}

\begin{proof}[Proof of claim]
The layer chain, regardless of its state, has the rate into $J$ bounded by $\Delta\sum_{j=1}^J q_j$. First note that for
\[
J_1:=\left\{j \in [J] : q_j\leq \frac{\theta }{\log N} \right\}
\]
we have, in a path to time $t \sim t_p$,
\[
\begin{split}
\p\left( J_1 \text{ ever visited} \right)
&\leq
\p \left( \operatorname{Bin} \left( t, \frac{\theta\Delta I}{\log N}\right)
\neq 0 \right)
=
O(\theta)
\end{split}
\]
and so these very small edge probabilities are not significant.

Similarly for the larger small probabilities we have Chernoff bound
\[
\begin{split}
\p\left( \frac{1}{\theta} \log \log N \text{ visits to } J_2\right)
&\leq
\p \left( \operatorname{Bin} \left( t, \frac{\theta\Delta I \log \log N}{\log N}\right)
\geq \frac{1}{\theta} \log \log N \right)\\
&\leq
\exp\left(
- \Omega\left( \frac{1}{\theta} \right)  \log \log N
\right).
\end{split}
\]

This gives a bound on the contribution to the $\log$ weight of the sample path determining $q^{(\alpha)}_{t-2\ell}(\varphi)$ by those $J_2$ states of
\[
\frac{1}{\theta} \log \log N \cdot \log \frac{\log N}{\theta}
\lesssim \frac{1}{\theta} \log^2 \log N.
\]

That is, a factor $e^{-O(\nicefrac{1}{\theta})\log^2 \log N}$ change in the path weight, with error probability (for the idealised layer chain path) $e^{-\Omega(\nicefrac{1}{\theta})\log \log N}.$%=\nicefrac{1}{\log^{\Omega(\nicefrac{1}{\theta})}N}$.
 
We can then upgrade the rerooting argument to discard  edges in $J_1$
\[
\min_{v \in [N]}
Q_{v,t}(\varphi)
\geq
\min_{v \in V}
\bar{Q}_{v,t-2\ell}(\varphi (\nicefrac{\theta}{\log N})^{-2\ell})
-
(\nicefrac{\theta}{\log N})^{\ell}
- 
\frac{\Delta I \theta\ell}{\log N}
\]
as by Markov's inequality at most $\nicefrac{\Delta I \theta\ell}{\log N}$ of the $\ell$-ball weight from any point goes to paths which feature a type $J_1$ edge.

As in the previous case (and with the same $\varphi
=\frac{1}{N} N^{\nicefrac{a}{\mu_p}} $), we still have
\[
\begin{split}
\mathcal{D}_{v}(t)
&\geq
\min_\alpha
q^{(\alpha)}_{t-2\ell}(\varphi)
-\frac{1+\frac{1}{\log N}}{
(\nicefrac{\theta}{\log N})^{\ell}\sqrt{
\varphi 
 N}}
 -(\nicefrac{\theta}{\log N})^{\ell}
 -\frac{\Delta I \theta\ell}{\log N}
 -
\frac{2}{\sqrt{\log N}}\\
&\gtrsim
\min_\alpha
\tilde{q}^{(\alpha)}_{t-2\ell- \theta^{-1} \log \log N}\left(\varphi e^{-O(\nicefrac{1}{\theta})\log^2 \log N} \right)\\
&\hspace{8em}-e^{-\Omega(\nicefrac{1}{\theta})\log \log N}
-\frac{1}{
(\nicefrac{\theta}{\log N})^{\ell}\sqrt{
\varphi 
 N}}
 -\frac{1}{\log^3 N}
 -
\frac{2}{\sqrt{\log N}}\\
&\gtrsim
\min_\alpha
\tilde{q}^{(\alpha)}_{t-2\ell- \theta^{-1} \log \log N}\left(\varphi e^{-O(\nicefrac{1}{\theta})\log^2 \log N} \right)
 -
\frac{2}{\sqrt{\log N}}
\end{split}
\]
where we now deal with the $J$ states separately and consider $\tilde{q}^{(\alpha)}_{t-2\ell- \theta^{-1} \log \log N}(\varphi)$ for the \emph{partially observed} chain on $[I]\setminus[J]$.
This partially observed chain $\tilde{L}$ has
\[
\forall i,k \notin [J] \quad
L_{ik} \leq
\tilde{L}_{ik} \leq L_{ik}+ \frac{\frac{\Delta I \theta \log \log N }{\log N} }{1-\frac{\Delta I \theta \log \log N}{\log N}},
\]
\[
\forall i \notin [J] \quad
\pi(i) \leq
\tilde{\pi}(i)
\leq \frac{\pi(i)}{1-\frac{\Delta I \theta \log \log N}{\log N} },
\]
and so modifies the mixing parameter of Corollary \ref{cor_contraction}
\[
\tilde{\delta}=\frac{1}{4\Delta}-O\left(\frac{\log \log N}{\log N}\right) \geq \frac{1}{5 \Delta}
\]
and we can again apply Theorem \ref{thm_concentration} to the partially observed chain, which now has $ \min_{i \notin [J]} q_i \log N \geq \theta \log \log N$ as in the previous case.
\end{proof}

%\begin{tcolorbox}
%$q_i$ yet smaller don't play a role
%\end{tcolorbox}
%
%\begin{proof}
%Now we can simply control $q^{(\alpha)}_{t-2\ell}(\varphi)$ on the event that these states are never visited, and the error probability for one layer can be controlled by
%\[
%1-
%I
%\left(
%1- \Delta N^{\nicefrac{-\theta}{\log\log N}}
%\right)^t
%\lesssim
%\frac{\Delta N^{\nicefrac{-\theta}{\log\log N}} \log N}{\mu_p}
%\]
%
%\end{proof}
%
%\begin{proof}%[Better proof]
%When $q_i < N^{\nicefrac{-\theta}{\log\log N}}$ we add some vector to produce $\min \tilde{q} = N^{\nicefrac{-\theta}{\log\log N}}$ and then couple to the unmodified chain, to time $t \sim t_p$, with probability at least
%\[
%\left(
%1-
%N^{\nicefrac{-\theta}{\log\log N}}
%\right)^{t_p}\rightarrow 1.
%\]
%\end{proof}

Claims 1--4 of this proof cover all cases for the vector $q$ apart from very large $q^*$ which was excluded before the first claim, and so by taking $\theta$ arbitrarily small we have proved Theorem \ref{thm_optimiser_cutoff}.
\end{proof}

\begin{proof}[Proof of Proposition \ref{prop_bad_set_control}]
Claims 3 and 4 of the previous proof give 
$\mathcal{D}(t) = 1-O(\theta)$ 
maximally over the set $\|\epsilon\|_{\infty}\geq \theta^2$.

However $p \notin \Pab$ requires $p_* \rightarrow 0$ and so $\epsilon\rightarrow 1$. By setting $\theta=\Theta(\delta)$ we can deduce
\[
\sup_{p \notin \Pab}
\mathcal{D}(t)
\geq
1-\delta.
\]
\end{proof}

{\bf Acknowledgements.}
This research was supported by NRDI grant KKP 137490.

%\nocite{lezaud01,cox62,bordenave18}

\printbibliography
%\bibliographystyle{apalike}
%\bibliography{multi_zb}

%\clearpage
\begin{appendices}
%\appendixpage
\noappendicestocpagenum
%\addappheadtotoc

\section{Appendix}\label{sec_appendix}

\begin{proposition}\label{prop_Qq_approx}
For any $t \leq \log^{\nicefrac{4}{3}} N$ and $p \in \mathscr{P}$, almost surely
\[
(1-\nicefrac{1}{\log N})
Q^{({ r}(p))}_{v,t}\left((1+\nicefrac{1}{\log N}) \varphi\right)
\leq
Q^{(p)}_{v,t}(\varphi)
\leq
(1+\nicefrac{1}{\log N})
Q^{({ r}(p))}_{v,t}\left((1-\nicefrac{1}{\log N}) \varphi\right)
\]
\[
(1-\nicefrac{1}{\log N})
Q^{(p)}_{v,t}\left((1+\nicefrac{1}{\log N}) \varphi\right)
\leq
Q^{({ r}(p))}_{v,t}(\varphi)
\leq
(1+\nicefrac{1}{\log N})
Q^{(p)}_{v,t}\left((1-\nicefrac{1}{\log N}) \varphi\right)
\]
and the same bounds hold deterministically for $q^{(\al)}$ for each $\al \in \D$.
\end{proposition}

\begin{proof}
As we noted in Remark \ref{remark_rel_error}, $\epsilon=\| 1-{ r}(p)/p  \|_\infty \leq \tfrac{1}{A}\log^{\nicefrac{-8}{3}}N$ and so we can control the relative error in the path weights
\[
(1-\epsilon)^t
\leq
\frac{\prod_{k=1}^{t} { r}(p)_{i_k}}{\prod_{k=1}^{t} p_{i_k}}
\leq (1+\epsilon)^t.
\]

From the expression for $Q$
\[
Q_{v,t}(\varphi)=
\sum_{w=1}^N
\sum_{\mathfrak{p}\in \mathcal{P}^t_{v,w}}
\mathbf{w}(\mathfrak{p})
\mathbbm{1}_{\mathbf{w}(\mathfrak{p})>\varphi}
\]
and the very similar version for $q$ (where $(M_k)_k$ is the layer chain \eqref{eq_layer_chain} from $\al$)
\[
q^{(\alpha)}_t(\varphi)=
\mathbb{P}\left(
\prod_{k=1}^{t} 
p_{M_k}
>\varphi
\right)
\]
we see
\[
(1-\epsilon)^t
Q^{({ r}(p))}_{v,t}\left((1+\epsilon)^t \varphi\right)
\leq
Q^{(p)}_{v,t}(\varphi)
\leq
(1+\epsilon)^t
Q^{({ r}(p))}_{v,t}\left((1-\epsilon)^t \varphi\right)
\]
which gives the claimed result as $\epsilon t \ll \nicefrac{1}{\log N}$.
\end{proof}

\begin{lemma}\label{lemma_banach}
The norm $\|\cdot\|$ forms a Banach algebra on $\mathbb{R}^n$, i.e. for every pair of vectors $f,g$ the pointwise product satisfies $\|fg\|\leq\|f\|\|g\|$.
\end{lemma}

\begin{proof}
As noted in \cite[Remark 2.6]{kloeckner2019}, it is sufficient to prove
\[
\sqrt{\operatorname{Var}_\pi(fg)}\leq \|f\|_\infty \sqrt{\operatorname{Var}_\pi(g)} + \|g\|_\infty \sqrt{\operatorname{Var}_\pi(f)}
\]
for which we first note
\[
\begin{split}
|f(i)g(i)-f(j)g(j)|
&=|f(i)(g(i)-g(j))+g(j)(f(i)-f(j))|\\
&\leq \|f\|_\infty |g(i)-g(j)| + \|g\|_\infty |f(i)-f(j)|.
\end{split}
\]

Then, we use this to expand a slightly unusual expression for the variance
\[
\begin{split}
\operatorname{Var}_\pi(fg)&=\frac{1}{2}\sum_{i=1}^n \sum_{j=1}^n (fg(i)-fg(j))^2\pi(i)\pi(j)\\
&\leq \|f\|_\infty^2 \operatorname{Var}_\pi(g) +\|g\|_\infty^2 \operatorname{Var}_\pi(f) 
+\|f\|_\infty\|g\|_\infty\sum_{i=1}^n \sum_{j=1}^n |f(i)-f(j)| |g(i)-g(j)|\pi(i)\pi(j)\\
&\leq \|f\|_\infty^2 \operatorname{Var}_\pi(g) +\|g\|_\infty^2 \operatorname{Var}_\pi(f) 
+\|f\|_\infty\|g\|_\infty\sqrt{4\operatorname{Var}_\pi(f)\operatorname{Var}_\pi(g)}
\end{split}
\]
where the final line uses the Cauchy-Schwarz inequality in $L^2(\pi\otimes\pi)$.
\end{proof}

\begin{lemma}\label{prop_max_entropy}
\[
\mu \leq \log(\Delta I).
\]
\end{lemma}

\begin{proof}
%We use that the uniform distribution has maximum entropy
\[
\begin{split}
\mu
&=
\sum_{\al \in \D} \frac{k(\al)}{N} \sum_{i=1}^I \al_i^+ p_i \log \frac{1}{p_i \al_i^+}
+\sum_{\al \in \D} \frac{k(\al)}{N} \sum_{i=1}^I \al_i^+ p_i \log \al_i^+\\
&\leq \log I + \log \Delta
\end{split}
\]
\end{proof}

\begin{lemma}\label{prop_max_square_entropy}
For any probability distribution $q_i$ on $[m]$ we have
\[
\sum_{i=1}^m q_i \log^2 q_i \leq 4 \vee \log^2 m
= \log^2 (m \vee 8).
\]
\end{lemma}

\begin{proof}
By the Lagrange multiplier we see that an optimiser can take either one value $q_i \equiv q$ or two values $q_i \in \{x,y\}$ satisfying
\[
xy =\frac{1}{e^2}
\implies
x\wedge y \geq \frac{1}{e^2}.
\]

This second condition (which is only possible when $m \leq 7$) then gives
\[
\sum_{i=1}^m q_i \log^2 q_i \leq \log^2 e^2=4
\]
and otherwise the uniform optimiser takes the other value as claimed.
\end{proof}

\begin{lemma}\label{prop_min_general_variance}
Given $A\leq \theta \log^{\nicefrac{1}{3}}N$ for some sufficiently small constant $\theta$, 
any $p \in \Pab$ has
\[
\mathbb{V}{\rm ar}_\pi \log p_i \geq A \log^{-\nicefrac{1}{3}} N.
\]
\end{lemma}

\begin{proof}
By a calculation in the proof of Corollary \ref{cor_min_variance}
\[
\begin{split}
\mathbb{V}{\rm ar}_\pi \log p_i 
&\geq
\frac{1}{\Delta^2} \log^2 \left( \frac{\Delta-1}{\Delta}  \right)
\wedge
\frac{p_*}{\Delta} \log^2 \left( p_* \Delta \right)\\
&\geq
\frac{A \log^{-\nicefrac{1}{3}} N}{\Delta} \log^2 \left( \frac{A \Delta }{\log^{\nicefrac{1}{3}} N} \right)
\geq
\frac{A \log^{-\nicefrac{1}{3}} N}{\Delta} \log^2 \left( \theta \Delta  \right)
\geq
A \log^{-\nicefrac{1}{3}} N
\end{split}
\]
using also in two places that $\theta$ is sufficiently small.
\end{proof}

%\begin{proposition}\label{prop_squared_log_hitting}
%An irreducible Markov chain $(X_t)_t$ with (finite) transition matrix $P$ on state space $S \ni x,y$ has
%\[
%\frac{\mathbb{E}_x\left(
%\sum_{t=1}^{T^+_y} f\left( X_{t-1}, X_t \right)
%\right)}{\mathbb{E}_x\left(T^+_y\right)}
%\leq
%\max_{s \in S} \left(\sum_{z \in S} P_{s,z} f(s,z) \right) \vee \left(-\sum_{z \in S} P_{s,z} \log P_{s,z}\right)
%\]
%where $T^+_y$ denotes the positive hitting time.
%\end{proposition}
%
%\begin{proof}
%Write for the row entropy
%\[
%H\left(P(s,\cdot)\right)=-\sum_{z \in S} P_{s,z} \log P_{s,z}
%\]
%and note that we can effectively increase the row entropy as desired by producing a coupled chain with a higher entropy contribution of one edge: for this edge $(s,z)$, create a copy of the vertex at $z$ which has the same probabilities out. Then the original edge can be given probability $\epsilon P_{s,z}$ and the new vertex has only one in-edge, from $s$, with probability $(1-\epsilon) P_{s,z}$. By identifying copies of a vertex we keep an almost sure coupling to the original chain.
%
%Thus, we can increase row entropies of the matrix $P$ until we have a matrix with $\emph{constant}$ row entropy and where this constant row entropy also exceeds the desired functional
%\[
%\max_{s \in S} \left(\sum_{z \in S} P_{s,z} f(s,z) \right)
%\]
%of the original chain -- i.e., where this constant row entropy is equal to the right-hand side of the claimed inequality.
%
%By applying \cite[Theorem 1.1]{choi18} we then have the result.
%\end{proof}

\begin{lemma}\label{lemma_covariance_and_separation}
\[
\mathbb{C}{\rm ov}\left(
f(X_0),f(X_t)
\right)
\leq
s(t) \mathbb{V}{\rm ar}_\pi\left(
f
\right)
\]
\end{lemma}

\begin{proof}
Because $P^t_{xy} \geq (1-s(t))\pi(y)$, there exists some stochastic matrix $E$ with
\[
\begin{split}
\mathbb{C}{\rm ov}\left(
f(X_0),f(X_t)
\right)
&=
\sum_{x,y}
f(x)f(y)\pi(x) P^t_{xy} -f(\pi)^2\\
&=
s(t)
\left(
\sum_{x,y}
f(x)f(y)\pi(x) E_{xy} -f(\pi)^2
\right)
\end{split}
\]
but also $\pi$ is stationary for $E$ and so we can bound the final covariance by $\mathbb{V}{\rm ar}_\pi\left(
f
\right)$.
\end{proof}

\begin{lemma}\label{prop_max_variance}
\[
\sigma^2_p \leq
2\Delta^2 \log^2 (I\vee 8).
\]
\end{lemma}

\begin{proof}
By Lemma \ref{lemma_dynamical_variance_bounds} and Proposition \ref{prop_max_square_entropy},
\[
\sigma^2_p\leq
(2\Delta -1)
\mathbb{E}_\pi(\log^2 p_i)\leq
(2\Delta -1)
\sum_{i=1}^I \Delta p_i \log^2 p_i\leq
2\Delta^2  \log^2 (I\vee 8).
\]
\end{proof}

\begin{lemma}\label{prop_quad_mu}
When $p$ is the optimiser of $\mu$ in $\mathscr{P}$, and $\epsilon_i:=|1-q_i/p_i|$ has $\|\epsilon\|_\infty\leq \tfrac{1}{2}$ we have
\[
\mu_p - \|\epsilon\|_{2,\pi}^2
\leq \mu_q
\leq \mu_p - \frac{1}{3} \|\epsilon\|_{2,\pi}^2.
\]
\end{lemma}

\begin{proof}
By assuming $p$ is an optimiser of $\mu$ in $\mathscr{P}$ we have zero first derivatives and second derivative as follows:
\[
\frac{\partial^2 \mu_q}{\partial q_i \partial q_j}\Bigg|_{q=p}
=
-\mathbbm{1}_{i=j}
\sum_{\al \in \D} \frac{k(\al)}{N} \cdot \frac{\al_i^+}{p_i}
%\leq -\mathbbm{1}_{i=j}
\]
which implies by Taylor's theorem with Lagrange remainder that for some $\eta_i$ between $p_i$ and $q_i$
\[
\mu_q= \mu_p-\frac{1}{2} \sum_{i=1}^I\sum_{\al \in \D} \frac{k(\al)}{N}  \frac{\al_i^+}{\eta_i} (q_i-p_i)^2
=\mu_p-\frac{1}{2}\mathbb{E}_\pi\left(
\frac{(q_i-p_i)^2}{\eta_i p_i}
\right)
\]
so we have the result after observing
\[
\frac{\|\epsilon\|_{2,\pi}^2}{1+\|\epsilon\|_\infty}\leq
\mathbb{E}_\pi\left(
\frac{(1-\nicefrac{q_i}{p_i})^2}{\nicefrac{\eta_i}{p_i}}
\right)
\leq \frac{\|\epsilon\|_{2,\pi}^2}{1-\|\epsilon\|_\infty}.
\]
\end{proof}

\end{appendices}

\end{document}